\documentclass{article}
\usepackage{amstext,amssymb, amsmath, graphicx, mathrsfs, amsthm, ulem}
\input xy
\xyoption{all}

\newcommand{\mO}{\mathcal{O}}
\newcommand{\p}{\mathfrak{p}}
\newcommand{\Q}{\mathfrak{P}}
\newcommand{\qB}{\mathfrak{Q}}
\newcommand{\q}{\mathfrak{q}}
\newcommand{\B}{\mathfrak{B}}

\newcommand{\Gal}{\mathrm{Gal}}
\newcommand{\Hom}{\mathrm{Hom}}
\newtheorem{thm}{Theorem}[section]
\newtheorem{cor}[thm]{Corollary}
\newtheorem{lem}[thm]{Lemma}
\newtheorem{pro}[thm]{Proposition}

\newtheorem{nota}[thm]{Notation} 

\theoremstyle{definition}

\newtheorem{rem}[thm]{Remark}
\newtheorem{exa}[thm]{Example}

\begin{document}
\setlength{\parindent}{15pt}
\setlength{\ULdepth}{1pt}

\title{\addtocounter{footnote}{1} A Multiquadratic Field Generalization of Artin's Conjecture, unabridged}
\author{Mar\'ia Elena Stadnik \footnote{Author partially supported by NSF grants
DMS-0701048 and DMS-0846285}\\
 Birmingham-Southern College \\}
\date{}
\maketitle

\tableofcontents

\begin{abstract} 
We prove that (under the assumption of the generalized Riemann hypothesis) a totally real multiquadratic number field $K$ has a positive density of primes $p$ in $\mathbb{Z}$ for which the image of $\mO_K^{\times}$ in $\left(\mO_K/p\mO_K\right)^{\times}$ has minimal index $ (p-1)/2$ if and only if $K$ contains a unit of norm $-1$.\\
\end{abstract}

\section{Introduction}
In 1927, Emil Artin \cite{Ar} conjectured that for any integer $a$ not equal to $\pm 1$ or a square, there is a positive density of primes $p$ for which $a$ is a primitive root in $\mathbb{F}_p^{\times}$.  In 1967, Hooley \cite{H} proved that this density exists and is positive if the generalized Riemann hypothesis is true.  Many authors have since adapted the analytic methods of Hooley to various settings, including Cooke and Weinberger \cite{CW}, Weinberger \cite{We}, Matthews \cite{KRM}, Cangelmi and Pappalardi \cite{CP},  Lenstra \cite{Le}, Murty \cite{Mu}, and Roskam \cite{R1}, \cite{R2}.   \\

This article deals with the following generalization of Artin's conjecture:\\

\textit{For which number fields $K/\mathbb{Q}$ is the index of the unit group $\mO_K^{\times}$ in $\left( \mO_K/p\mO_K\right)^{\times}$  precisely $ (p-1)/2$ for infinitely many primes $p$?}\\

For any prime $p$ completely split in $K/\mathbb{Q}$, let $\psi_p$ denote the map $(\mO_K)^{\times} \rightarrow \left( \mO_K/p\mO_K\right)^{\times}$.  Since the image of $\mO_K^{\times}$ lies in the kernel of the norm map $N:\left( \mO_K/p\mO_K\right)^{\times} \rightarrow \mathbb{F}_p^{\times}/\left\{\pm 1\right\}$, the smallest possible index of $\psi_p\left(\mO_K^{\times}\right)$ in $\left( \mO_K/p\mO_K\right)^{\times}$ is $ (p-1)/2$.  We say that $K$ has \textit{minimal index mod $p$} if this index is exactly $ (p-1)/2$.  If every unit of $K$ has norm $+1$, then $\psi_p\left(\mO_K^{\times}\right)$ is contained in the kernel of the map $N:\left( \mO_K/p\mO_K\right)^{\times} \rightarrow \mathbb{F}_p^{\times}$.  Thus the index of $\psi_p(\mO_K^{\times})$ in $\left( \mO_K/p\mO_K\right)^{\times}$ is at least $p-1$.  We immediately deduce that if our generalization of Artin's conjecture is to be true, $K$ must contain a unit of norm $-1$.  We also quickly deduce a second necessary condition on $K$.\\

\noindent \textbf{Proposition \ref{proposition:totreal}}
\textit{If $K$ is a number field that has minimal index mod $p$ for infinitely many primes $p$ split in $K/\mathbb{Q}$, then $K$ is a totally real number field.}\\

 In this paper,  we adapt the techniques of Hooley to prove the following theorem.\\

\noindent \textbf{Theorem \ref{maintheorem}}
\textit{Let $K$ be a totally real multiquadratic field with unit group $\mO_K^{\times}$ that has ramification index $e_2(K/\mathbb{Q}) \leq 2$.  The generalized Riemann hypothesis implies that there is a positive
density of primes split in $K/\mathbb{Q}$ for which the image of $\mO_K^{\times}$ in $\left( \mO_K/p\mO_K\right)^{\times}$ has index $ (p-1)/2$ if and only if $K$ contains a unit of norm $-1$.   Moreover, when $K$ contains a unit of norm $-1$, the density $c$ is given by
$$ c= \frac{1}{2[K:\mathbb{Q}]} \left( \sum_{k|\Delta} \frac{\mu(k)|C_k|}{[K_k:K]} \right)\prod_{l \nmid 2\Delta}  \left(1 - \frac{1}{(l-1)}\left(1-\frac{(l-1)^{n-1}}{l^{n-1}}\right)\right).$$}

As an example, consider the field $K=\mathbb{Q}(\sqrt{5}, \sqrt{13})$, which satisfies the hypotheses of the theorem.  Using PARI, we computed the density of primes $p$ split in $K/\mathbb{Q}$ for which $K$ has minimal index mod $p$ among the first two hundred thousand primes to be 
$$ 0.05176. $$
In this paper, we prove that the generalized Riemann hypothesis implies the density of such primes $p$ should be $$ \frac{1}{8} \prod_{l \neq 2} \left(1- \frac{1}{[K(\zeta_l):K]}\left(1-\frac{(l-1)^3}{l^3}\right)\right) \approx 0.05142184\ldots,$$
which is very close to our calculations.  \\

We conjecture that Theorem \ref{maintheorem} holds for any totally real Galois number field $K$.  We also consider the related problem of determining which multiquadratic fields contain a unit of norm $-1$.  \\

\noindent \textbf{Theorem \ref{nonorm-1}}
\textit{If $K=\mathbb{Q}(\sqrt{d_{1}}, \ldots ,\sqrt{d_{m}})$ is a multiquadratic field of degree $n=2^m$, $m \geq 3$ and the class number of $K$ is odd, then $K$ contains no unit of norm $-1$. }\\

\textbf{Relationship to Artin-type problems.} As an integer $a$ is a primitive root in $\mathbb{F}_p^{\times}$ if and only if the subgroup that $a$ generates has smallest index $1$ in $\mathbb{F}_p^{\times}$, we see that this question is a natural generalization of Artin's conjecture.  Other Artin type problems, such as Weinberger's conjectural proof \cite{We} that the ring of integers of a number field with infinitely many units is a Euclidean domain if it is a principal ideal domain, become easier as the field degree increases.  This is partly the case in the setting of this problem, because as the field degree $n=[K:\mathbb{Q}]$ increases, so does the rank of the unit group $\mO_K^{\times}$ since $K$ is totally real.  However, there is a difficulty in our case that does not arise in most Artin-type problems.  Namely, as the field degree $n$ increases, the condition we impose requires that the image of the unit group $\mO_K^{\times}$ generates modules with increasing cardinality.  More precisely, for odd primes $l|p-1$, we require that the image of the unit group (which has rank $n-1$) generates a vector space in $\displaystyle \prod_{n} \left(\mathbb{F}_p^{\times}/(\mathbb{F}_p^{\times })^l\right)$ of rank $n-1$.\\
  
\textbf{Restriction to multiquadratic fields.}    We investigate only number fields that are Galois over $\mathbb{Q}$ in order to utilize the structure of the Galois group of $K/\mathbb{Q}$ in our proofs.  The assumption that $K$ is multiquadratic is essential to our arguments.  We use this assumption at various points in the article with the key stumbling block being Lemma \ref{M3lemma}.  Lemma \ref{M3lemma} is the generalization of Hooley's argument of sieving in the range of the largest primes.  To prove this lemma, we find a way to deal with all units of $K$ simultaneously by finding a nice basis of units $\epsilon_1, \ldots, \epsilon_{n-1}$ such that if the orders of $\epsilon_i$ modulo $p$ are coprime to $l$, then the orders of \textit{all} units are also coprime to $l$.  For example, if neither $\epsilon_i$ nor $\epsilon_j$ are $l^{th}$ powers, then neither is any multiplicative combination of the $\epsilon_i$ and $\epsilon_j$.  The way we do this is to diagonalize the units under the action of $G=\Gal(K/\mathbb{Q})$.  This approach is successful provided that two conditions hold.  First, this decomposition must split the units into $G$-representations of dimension $1$.  Because of the $G$-action on the units, this is only possible when $G$ is abelian.  Second, we require that this splitting occurs over $\mathbb{F}_l$ for \textit{every} prime $l$.  For both conditions to be satisfied, $G$ must be isomorphic to $(\mathbb{Z}/2\mathbb{Z})^m$ for some $m$, or equivalently $K$ must be multiquadratic.  \\

\textbf{Generalizations.}  In this article we choose only to investigate primes $p\in \mathbb{Z}$ that are split completely in $K/\mathbb{Q}$.  This is done to simplify many calculations and make concrete the structure of $\left( \mO_K/p\mO_K\right)^{\times}$ for general Galois number fields $K$.  It is expected that our arguments can be generalized to count densities of primes $p$ not split in $K/\mathbb{Q}$ in a manner similar to  \cite{R2}.  Primes not split in $K/\mathbb{Q}$ must be treated on a case by case basis for each Galois group.  There is some possibility that we could restrict ourselves to sieving primes $p$ such that all divisors of $p-1$ are primitive roots modulo $N$.  This would allow us to generalize the results of this paper to abelian extensions with Galois group $\mathbb{Z}/N\mathbb{Z}$ for certain $N$.  There seems to be some general difficulties to this approach.  \\
  
\textbf{Organization.} This article is organized as follows.  Section \ref{section:reformulation} includes a reformulation of the question into a question about a set of Frobenius elements in field extensions $K_l /K$ as $l$ ranges over all prime numbers that holds for general Galois number fields $K$ (as long as a technical condition, \textbf{(G1)}, is satisfied; see Remark \ref{G1}).  Section \ref{mq} includes a further reformulation of the problem for multiquadratic fields.  The techniques of Hooley are adapted in Section \ref{proofs} to prove the theorem for multiquadratic fields.  Section \ref{examples}  includes an investigation of units in multiquadratic fields.  Finally, Section \ref{rayclassfields} provides an application to ray class fields of conductor $p\mO_K$.

\section{Reformulation of the problem}\label{section:reformulation}
Let $K \neq \mathbb{Q}$ be a Galois number field containing a unit of norm $-1$.  Let $G=\Gal(K/\mathbb{Q})$ and $n=|G|=[K:\mathbb{Q}]$. 
Let $\mO_K$ denote the ring of algebraic integers of $K$, and let $\mO_K^{\times}$ denote the group of units of $\mO_K$, which we usually refer to as \textit{units of K}.  

\subsection{Restrictions on the field $K$}
Not every Galois number field $K$ containing a unit of norm $-1$ will have have minimal index mod $p$ for a positive density of split primes $p$.  We quickly deduce a necessary condition for this to be the case.

\begin{pro}\label{proposition:totreal}
If $K$ is a number field that has minimal index mod $p$ for infinitely many primes $p$ split in $K$, then $K$ is a totally real number field.
\end{pro}

\begin{proof}
Suppose $K$ has signature $(r,s)$ and that $K$ has minimal index mod $p$ for infinitely many primes $p$ split in $K/\mathbb{Q}$.  Let $t$ be the rank of $\mO_K^{\times}$, so that $\mO_K^{\times} \cong \mathbb{Z}^t \oplus (\mO_K^{\times})_{\text{tors}}$, and let $m=|(\mO_K^{\times})_{\text{tors}}|<\infty$.  $\mO_K^{\times}$ has rank $t= r + s -1$.  By assumption we may consider a prime $p > m + 1$ completely split in $K$ for which $K$ has minimal index mod $p$; this implies $|\psi_p(\mO_K^{\times})|=2(p-1)^{n-1}$. Since $\psi_p$ is a map of $\mathbb{Z}-$modules, the maximal size of the image of $\psi_p$ is $m (p-1)^t$, so $m (p-1)^t= 2(p-1)^{n-1}$.  Since $ m < p-1$, this implies $t=n-1$.
\end{proof}
 
From now on we add the assumption that $K$ is totally real. 
\subsection{Transition from $\psi$ to $\phi_l$ for primes $l | p-1$}

Throughout this section, fix a prime $p \in \mathbb{Z}$ which splits completely in $K/\mathbb{Q}$, and let $\{ \p=\p_{1}, \p_{2}, ..., \p_{n} \} $ be the set of prime ideals of $K$ lying over $p$.   We will suppress the subscript $p$ from the map $\psi_p$.  $p$ is completely split in $K$ and $G$ acts transitively and faithfully on the set $\{ \p=\p_{1}, \p_{2}, \ldots , \p_{n} \} $ of prime ideals of $K$ lying over $p$, so for each $1 \leq i \leq n$, there is a unique $g=g_{i} \in G$ such that $g_{i}(\p) = \p_{i}$.  Define $\phi$ to be the composition 
 $$ \mO_K^{\times} \stackrel{\psi}{\rightarrow}  (\mO_{K}/p\mO_{K})^{\times} \stackrel{\sim}{\rightarrow} \prod_{i=1}^n (\mO_K/\p_i)^{\times} \stackrel{\sim}{\rightarrow} \prod_{i=1}^n \mathbb{F}^{\times}_p \stackrel{\sim}{\rightarrow} \mathbb{Z}/(p-1)\mathbb{Z}[G], $$
 $$ u \mapsto (a^{t_1}, \ldots, a^{t_n}) \mapsto \sum_{i=1}^n t_i [g_i]=: \phi(u) $$
where $a$ is a choice of a generator of $\mathbb{F}^{\times}_p$.  The image of $\psi$ lies in the kernel of the composition $(\mO_K/p\mO_K)^{\times} \stackrel{N}{\rightarrow} \mathbb{F}_p^{\times} \rightarrow \mathbb{F}_p^{\times}/\{\pm 1\}$, which is isomorphic as a $G$-module to the additive module $I=\ker(\mathbb{Z}/(p-1)\mathbb{Z}[G] \rightarrow \mathbb{Z}/((p-1)/2)\mathbb{Z}) \subseteq \mathbb{Z}/(p-1)\mathbb{Z}[G]$. Thus the image of $\phi$ lies in $I$. Notice $|I|=2(p-1)^{n-1}$ and $|Im(\psi)|= |Im(\phi)|$.

\begin{lem}\label{z2iso}
$K$ has minimal index mod $p$ if and only if for every prime $l$ there is an isomorphism of $\mathbb{F}_l[G]$-modules $\mO_K^{\times}/(\mO_K^{\times})^l \stackrel{\sim}{\rightarrow} I/lI$.  
\end{lem}

\begin{proof}
Suppose that $K$ has minimal index mod $p$, so that $|Im(\psi)|=2(p-1)^{n-1}$.  Since $2(p-1)^{n-1}=|I|$, $\phi : \mO_K^{\times} \rightarrow I$ is an epimorphism.  For any $l|p-1$, tensoring this map with $\mathbb{Z}/l\mathbb{Z}$ gives an epimorphism of $\mathbb{F}_l[G]$-modules $\mO_K^{\times}/(\mO_K^{\times})^l \twoheadrightarrow I/lI$.   Both $\mO_K^{\times}/(\mO_K^{\times})^l$ and $I/lI$ are $\mathbb{F}_l[G]$-modules of rank $n-1$ if $ l \neq 2$ or $\mathbb{F}_2[G]$-modules of rank $n$ if $l=2$, this map must be an isomorphism. If $l \nmid p-1$, then both $\mO_K^{\times}/(\mO_K^{\times})^l$ and $I/lI$ are 0 and the map is an isomorphism. \\

Now suppose that for every prime $l$ there is an isomorphism $\mO_K^{\times}/(\mO_K^{\times})^l \stackrel{\sim}{\rightarrow} I/lI$.  Let $M$ denote the cokernel $M=I/\phi(\mO_K^{\times})$; we show $M=0$ by showing for any $l$, the localized module $M_l=0$.  By definition, the sequence 
$$\mO_K^{\times} \stackrel{\phi}{\rightarrow} I \rightarrow M \rightarrow 0$$
is right exact, and so the sequence
$$\mO_K^{\times}/(\mO_K^{\times})^l \rightarrow I_l/lI_l \rightarrow M_l/lM_l \rightarrow 0$$
is also right exact.  By assumption the first map in this sequence is an isomorphism, so by exactness $M_l/lM_l=0$.  Hence $M_l = 0$ by Nakayama's lemma.
\end{proof}

The additive module $I$ fits into the exact sequence of $\mathbb{Z}-$modules
$$ 0 \longrightarrow I \longrightarrow \mathbb{Z}/(p-1)\mathbb{Z}[G] \stackrel{Tr}{\longrightarrow} \mathbb{Z}/((p-1)/2)\mathbb{Z} \longrightarrow 0,$$
where $Tr$ denotes the trace map that sends an element $\displaystyle \sum_{g \in G} a_g[g]$ to $\displaystyle \sum_{g \in G} a_g \in  \mathbb{Z}/\left((p-1)/2\right)\mathbb{Z}$.  \\
Let $B_l=\{x=\sum x_{i}[g_{i}] \in \mathbb{F}_{l}\left[G\right]:\sum x_{i}=0 \} \subseteq \mathbb{F}_l[G]$.

\begin{lem} 
$I/lI \cong B_l$ if $ l \neq 2$.
\end{lem}

\begin{proof}
Apply $\otimes_{\mathbb{Z}}\mathbb{Z}/l\mathbb{Z}$ to the above sequence and notice that some of the terms in the long exact sequence of Tor are zero.
\end{proof}

For $l=2$ the situation is more complicated.
\begin{lem}\label{torat2}
\begin{enumerate}\renewcommand{\labelenumi}{(\alph{enumi})}
\item If $p \equiv 3 \mod 4$, then $I/2I \cong \mathbb{F}_2[G]$.
\item If $ p \equiv 1 \mod 4$ and $|G|$ is odd, then $I/2I \cong \mathbb{F}_2[G]$.
\item If $p \equiv 1 \mod 4$ and $|G|$ is even, then $I/2I$ is not isomorphic to $\mathbb{F}_2[G]$.
\end{enumerate}
\end{lem}
\begin{proof}
$(a):$ Similar to the proof of the previous lemma.\\

$(b)$, $(c):$  Suppose $p \equiv 1 \mod 4$. As an abelian group, $I \cong (\mathbb{Z}/(p-1)\mathbb{Z})^{n-1} \oplus \mathbb{Z}/2\mathbb{Z}$, and so  $I/2I$ is a rank $n$ module over $\mathbb{Z}/2\mathbb{Z}$.  Let $d$ be such that $2^d ||p-1$; note $d \geq 2$.  Define $T=\ker(\mathbb{Z}/2^d\mathbb{Z}[G] \rightarrow \mathbb{Z}/(2^{d-1})\mathbb{Z})$, and define a $G$-module homomorphism $\delta: I \rightarrow T$ by reducing coefficients of elements of $I \mod 2^{d}$ and acting by the identity on elements of $G$.  Using the generalized Chinese remainder theorem we deduce that this map is an epimorphism. Tensoring with $\mathbb{Z}/2\mathbb{Z}$ we obtain a $G$-module map $I/2I \twoheadrightarrow T/2T$ 
between $G$-modules over $\mathbb{F}_2$ of the same rank, thus an isomorphism.\\

First suppose that $|G|$ is odd.  Notice every element in $T$ can be written as $2^{d-1} \beta + \sum_{i=1}^n a_i [g_i]$, where $\sum_{i=1}^n a_i \equiv 0 \mod 2^d$.  Let $\alpha= (2^{d-1} - |G|)[g_1] + \sum_{i=1}^n [g_1] \in T$; then for any $g_j \in G$, $g_j \alpha = \alpha + (2^{d-1}-|G|)([g_j]-[g_1])$.  Since $2^{d-1}-|G|$ is odd by assumption, the set $\left\{g_j \alpha - \alpha\right\}_{i=1}^n$ generates all elements  $\sum_{i=1}^n a_i [g_i]$ with $\sum_{i=1}^n a_i \equiv 0 \mod 2^d$ in $T/2T$.  Since $\alpha$ itself has trace $2^{d-1}$, we see that $\alpha$ generates all of $T/2T$ as a $G$-module.  Thus $T/2T$ is cyclic as a $G$-module; since it is of rank $n$, it must be isomorphic to $\mathbb{F}_2[G]$, and thus also $I/2I\cong \mathbb{F}_2[G]$.  \\

Now suppose that $|G|$ is even.  The element $\alpha$ lies in the $G$ invariants of $T/2T$ since for any $g_j \in G$, $g_j \alpha - \alpha= (2^{d-1}-|G|)([g_j]-[g_1]) \in 2T$.  Since also $2^{d-1}[g_1]$ lies in the $G$ invariants of $T/2T$, the $G$ invariants of $T/2T$ is at least $2-$dimensional over $\mathbb{F}_2$.  Since the $G$ invariants of $\mathbb{F}_2[G]$ is $1-$dimensional, $T/2T$ is not isomorphic to $\mathbb{F}_2[G]$ as $G$-modules, and so $I/2I$ is not isomorphic to $\mathbb{F}_2[G]$ as $G$-modules.

\end{proof}

For any $l$ (except $l=2$ when $|G|$ is even), let $\phi_l$ denote the composition of $\phi$ with the map to $\mathbb{F}_l[G]$ obtained from the previous lemmas.  Explicitly, this map is simply reduction mod $l$ of the coefficients $t_i$ of elements of $\mathbb{Z}/(p-1)\mathbb{Z}[G]$.  Using the previous three lemmas, we deduce the following theorem.
  
\begin{thm}\label{relatepsiandbl}
If $p \equiv 3 \mod 4$ or if $p \equiv 1 \mod 4$ and $|G|$ is odd, then $K$ has minimal index mod $p$ if and only if the image of  $\phi_l$ equals $B_l$ for every $l|p-1$, $l\neq 2$ and the image of $\phi_2$ is all of $\mathbb{F}_2[G]$.  If $ p \equiv 1 \mod 4$ and $|G|$ is even, then $K$ has minimal index mod $p$ if and only if $\mO_K^{\times}/(\mO_K^{\times})^2 \stackrel{\sim}{\rightarrow} I/2I$ and the image of  $\phi_l$ equals $B_l$ for every $l|p-1$, $l \neq 2$.   
\end{thm}

We also deduce the following weaker theorem, which will be useful later.

\begin{thm}\label{powerof2}
The index of $\psi_p(\mO_K^{\times})$ in $\left( \mO_K/p\mO_K\right)^{\times}$ is a power of $2$ times $ (p-1)/2$ if and only if the image of  $\phi_l$ equals $B_l$ for every $l|p-1$, $l\neq 2$. 
\end{thm}

\subsection{Transition from $\phi_l$ to Frobenius elements}

Theorem \ref{relatepsiandbl} can be reformulated as a question about certain Frobenius elements attached to the $\phi_{l}$ for $l|p-1$.  Given a prime $l |p-1$,  let $K_{l}=K\left(\zeta_{l}, \sqrt[l]{\mO_K^{\times}}\right)$, where $\zeta_l$ denotes a primitive $l^{th}$ root of unity.
Then $K_{l}$ is a Galois extension of $K(\zeta_{l})$, so let $G_{l}=\Gal(K_{l}/K(\zeta_{l}))$.  Since $\mO_K^{\times}$ is a $\mathbb{Z}-$module of rank $n-1$ with $\left(\mO_K^{\times}\right)_{\text{tors}}= \left\{\pm 1 \right\}$, for $l\neq 2$, $|G_l|= l^{n-1}$ and $G_{l}\cong (\mathbb{Z}/l\mathbb{Z})^{n-1}$ as groups.  In contrast, for $l=2$, $\sqrt{-1}$ is in $ K_2$, but not in $K(\zeta_2)=K$; thus $G_2 = \Gal(K_2/K(\zeta_2))\cong (\mathbb{Z}/2\mathbb{Z})^{n} $ and $|G_2|=2^n$.  \\

The next few lemmas will prove that the fields $K_l$ can be generated as a $G$-module over $K(\zeta_l)$ by a single unit.  The following lemma is proved in Washington \cite{cyclo}.

\begin{lem}\label{Washington} \cite[Lemma 5.27]{cyclo}
There is a unit $\epsilon$ of $K$ such that the set of units $\left\{g_i(\epsilon)\right\}_{i=1}^{n-1}$ is multiplicatively independent, hence is such that $\left[\mO_K^{\times}: \mathbb{Z}[G]\cdot \epsilon \right]=N <\infty$.
\end{lem}

The next lemma says that for every $l \nmid n=|G|$ (in particular the finitely many $l|N$ that are coprime to $n$), there exists a unit $\epsilon_l$ such that $\gcd \left(\left[\mO_K^{\times}: \mathbb{Z}[G]\cdot \epsilon_l \right], l \right) = 1$.

\begin{lem}\label{lnmidn}
If $l \nmid n$, then there is a unit $\epsilon_l$ so that $\gcd \left(\left[\mO_K^{\times}: \mathbb{Z}[G]\cdot \epsilon_l \right], l \right) = 1$.
\end{lem}
\begin{proof}
We first show that $\mO_K^{\times} \otimes \mathbb{F}_l$ is cyclic as a $G$-module over $\mathbb{F}_l$.  Consider the map
$ \mO_K^{\times} \stackrel{L}{\longrightarrow} \mathbb{R}[G],$
$L: u \mapsto (\log|g_1(u)|, \log|g_2(u)|, \ldots ,\log|g_n(u)|)$.  $L$ maps $\mO_K^{\times}$ into the trace zero submodule $\mathbb{R}[G]^{\text{tr}=0}$ of $\mathbb{R}[G]$.  By Dirichlet's unit theorem, after tensoring with $\mathbb{R}$ this map becomes an isomorphism $\mO_K^{\times}\otimes_{\mathbb{Z}} \mathbb{R} \stackrel{L \otimes_{\mathbb{Z}} \mathbb{R}}{\longrightarrow} \mathbb{R}[G]^{\text{tr}=0}$.  Since $\mathbb{C}$ is flat over $\mathbb{R}$, this gives an isomorphism $\mO_K^{\times}\otimes \mathbb{C}\cong \mathbb{C}[G]^{\text{tr}=0}$.  By considering the $G \times G$ bimodule structure on $\mathbb{Q}[G]$, it follows that $\mO_K^{\times}\otimes \mathbb{Q}$ is a direct summand of $\mathbb{Q}[G]$ and hence cyclic as a $\mathbb{Q}[G]$-module.  Because localization is exact, $\mO_K^{\times}\otimes \mathbb{Q}_l$ is cyclic as a $\mathbb{Q}_l[G]$-module.  Since $l \nmid |G|$, there is an equivalence of categories between $\mathbb{Q}_l[G]$-modules and $\mathbb{F}_l[G]$-modules, and so $\mO_K^{\times} \otimes \mathbb{F}_l $ is also cyclic as an $\mathbb{F}_l[G]$-module.  \\
   
Since $\mO_K^{\times} \otimes \mathbb{F}_l $ is cyclic as a $G$-module, by Nakayama's lemma the $\mathbb{Z}_{(l)}-$module $\mO_K^{\times} \otimes \mathbb{Z}_{(l)}$ obtained by tensoring $\mO_K^{\times}$ with the local ring $\mathbb{Z}_{(l)}$ is also cyclic as a $G$-module.  Thus there exists a unit $\tilde{\epsilon_l} \in \mO_K^{\times}\otimes \mathbb{Z}_{(l)}$ such that $\mO_K^{\times}\otimes \mathbb{Z}_{(l)}=\mathbb{Z}_{(l)}[G]\cdot \tilde{\epsilon_l}$.  Viewing $\mathbb{Z}_{(l)} \subseteq \mathbb{Q}$, we see $\tilde{\epsilon_l}=u^{r/s}$ for some $u \in \mO_K^{\times}$ and integers $r$ and $s$ with $l \nmid s$.  Set $\epsilon_l := u^r \in \mO_K^{\times}$; then the $G$-modules generated by $\epsilon_l$ and $\tilde{\epsilon_l}$ coincide in $ \mO_K^{\times}\otimes \mathbb{Z}_{(l)}$.  Thus after tensoring the short exact sequence 
$$ 0 \rightarrow \mathbb{Z}[G]\stackrel{\cdot \epsilon_l}{\longrightarrow} \mO_K^{\times} \rightarrow  \mO_K^{\times}/\mathbb{Z}[G]\cdot \epsilon_l \rightarrow 0$$
with $\mathbb{Z}_{(l)}$, we obtain the exact sequence
$$  \mathbb{Z}[G]\otimes \mathbb{Z}_{(l)}\stackrel{\cdot \epsilon_l}{\longrightarrow} \mO_K^{\times}\otimes \mathbb{Z}_{(l)} \rightarrow  \mO_K^{\times}/\mathbb{Z}[G]\cdot \epsilon_l\otimes \mathbb{Z}_{(l)} \rightarrow 0,$$
where the first map is an isomorphism.  Thus $\mO_K^{\times}/\mathbb{Z}[G]\cdot \epsilon_l\otimes \mathbb{Z}_{(l)} =0$.  Hence $\mO_K^{\times}/\mathbb{Z}[G]\cdot \epsilon_l$ is torsion as a $\mathbb{Z}[G]$-module and has order coprime to $l$.
\end{proof}

 \begin{rem}\label{G1} It remains to consider those primes $l|n$.  The following condition must be satisfied by the field $K$:\\
$$\text{\textbf{(G1)} If }l|n\text{, then there is a unit }\epsilon_l \in \mO_K^{\times}\text{ such that }~\gcd \left(\left[\mO_K^{\times}: \mathbb{Z}[G]\cdot \epsilon_l \right], l\right)=1.$$

We prove in Theorem \ref{G1satisfied} that this condition is satisfied in a multiquadratic field by any unit that has norm $-1$.  For more general totally real Galois fields this condition must be satisfied; it is critical for the strategy of our proof.

\end{rem}

For a unit $\eta \in \mO_K^{\times}$, let $U_{\eta} = \mathbb{Z}[G]\cdot \eta$ be the multiplicative subgroup of $\mO_K^{\times}$ generated by $\eta$ as a $G$-module.  

\begin{lem}\label{lthroots}
If $\eta$ is a unit satisfying  $\gcd\left(\left[\mO_K^{\times}: U_{\eta}\right], l \right)=1$ then $\mO_K^{\times} \cdot \sqrt[l]{U_{\eta}} = \sqrt[l]{\mO_K^{\times}}$.
\end{lem}

\begin{proof}
$U_{\eta} \subseteq \mO_K^{\times}$ and so $\mO_K^{\times} \cdot \sqrt[l]{U_{\eta}} \subseteq \sqrt[l]{\mO_K^{\times}}$.  Conversely, let $\left[\mO_K^{\times}: U_{\eta}\right]=r$ and let $x \in \sqrt[l]{\mO_K^{\times}}$; notice that $x^r \in \sqrt[l]{U_{\eta}}$ and $x^l \in \mO_K^{\times}$.  Since $\gcd(r,l)=1$ there are integers $a$ and $b$ such that $1=al+br$, and $\displaystyle x=x^{al+br}=(x^l)^b(x^r)^a \in \mO_K^{\times} \cdot \sqrt[l]{U_{\eta}}$.
\end{proof} 

For a unit $\eta$, set $\eta_{i}:= g_{i}^{-1}(\eta)$ for $1 \leq i \leq n $.  By Lemmas \ref{Washington}, \ref{lnmidn}, and \ref{lthroots}, and because $K$ satisfies condition \textbf{(G1)}, we deduce for every prime $l$ there is a unit $\epsilon$ such that  $K_{l}=K(\zeta_{l}, \sqrt[l]{\epsilon_{1}}, \ldots , \sqrt[l]{\epsilon_{n}})$.  Also, note there are only finitely many distinct $\epsilon$ for all $l$ prime.\\

 Given a prime $\Q$ of $K(\zeta_l)$ lying above $\p \subseteq \mO_K$, let $\sigma_{\Q}=(\Q, K_l/K(\zeta_l)) \in G_{l}$ be the Frobenius element attached to a prime $\Q$.  This is a well defined element of $G_{l}$ (independent of choice of prime ideal $\B|\Q$ in $K_l$) since the extension $K_{l}/K(\zeta_{l})$ is abelian. 

\begin{thm}\label{reformulation}
Assume that $K$ satisfies condition \textbf{(G1)}.  Let $p \in \mathbb{Z}$ be a completely split prime in $K$. 
\begin{enumerate}\renewcommand{\labelenumi}{(\alph{enumi})}
\item If $p \equiv 3 \mod 4$, $K$ has minimal index mod $p$ if and only if for every $l|p-1$, the Frobenius element $\sigma_{\Q}$  of any prime $\Q$ of $K(\zeta_l)$ lying above $p$ generates $G_{l}=\Gal(K_{l}/K(\zeta_{l}))$ as a $G$-module.  
\item If $p \equiv 1 \mod 4$ and $|G|$ is even, then  $K$ has minimal index mod $p$ if and only if  $\mO_K^{\times}/(\mO_K^{\times})^2 \stackrel{\sim}{\rightarrow} I/2I$ and for every $l|p-1$, $l \neq 2$, the Frobenius element $\sigma_{\Q}$ of any prime $\Q$ of $K(\zeta_l)$ lying above $p$ generates $G_l$ as a $G$-module.  
\item If $p \equiv 1 \mod 4$ and $|G|$ is odd, then $K$ does not have minimal index mod $p$.
\end{enumerate}
\end{thm}

\begin{proof}
  For any $l|p-1$, let $G'_l:=\Gal(K_l/K)$.  From Kummer theory, there is a short exact sequence
$$ 1 \rightarrow \mu_l \rightarrow \bar{K}^{\times} \rightarrow \bar{K}^{\times} \rightarrow 1$$
given by sending an element of $\bar{K}^{\times}$ to its $l^{th}$ power, where $\mu_l$ denotes the set of $l^{th}$ roots of unity.  Taking invariants under the absolute Galois group $\Gal(\overline{K(\zeta_l)}/K(\zeta_l))$ of $K(\zeta_l)$ and noticing that $H^1(K(\zeta_l), \bar{K}^{\times})=0$ by Hilbert's theorem 90, we obtain the Kummer isomorphism $H^1(K(\zeta_l), \mu_l ) \cong K(\zeta_l)^{\times}/K(\zeta_l)^{\times l}$.  Since $\mu_l$ is trivial over $K(\zeta_l)$, $H^1(K(\zeta_l), \mu_l)$ records the homomorphisms from the absolute Galois group of $K(\zeta_l)$ to $\mu_l$.  This is dual to the Galois group $\Gamma_{K(\zeta_l)}$ of the maximal abelian extension of $K(\zeta_l)$, and so we have
$$ \Gamma_{K(\zeta_l)}\cong  \Hom(K(\zeta_l)^{\times}/K(\zeta_l)^{\times l} , \mu_l) = \Hom(K(\zeta_l)^{\times}, \mu_l)$$
given by sending $\sigma$ to $\displaystyle \left(\alpha \mapsto \frac{\sigma(\sqrt[l]{\alpha})}{\sqrt[l]{\alpha}}\right)$.  For any place $v$ in $K(\zeta_l)$ above $p$, we have the commutative diagram\\
\begin{center}
$ \xymatrix{   \Gamma_{K(\zeta_l)_v} \ar[d] \ar[r] & \Hom(K(\zeta_l)^{\times}_v , \mu_l)  \ar[d]   \\ 
\Gamma_{K(\zeta_l)} \ar[r] & \Hom(K(\zeta_l)^{\times}, \mu_l)}$ \\
\end{center}
where the vertical maps are isomorphisms.  Let $k$ be the residue field of $K(\zeta_l)$ at $v$; by assumption on $p$, this is also the residue field of $K$ at $\p$.  Then there is a map
$$\Gamma_{K(\zeta_l)_v}= \Hom(K(\zeta_l)^{\times}_v, \mu_l) \rightarrow \Hom(\mO_{l,v}^{\times}, \mu_l) \rightarrow \Hom(k^{\times}, \mu_l)$$
which corresponds to the maximal unramified quotient of $\Gamma_{K(\zeta_l)_v}$.  The image of a lift of the Frobenius element is computed on $\alpha \in \mO_{l,v}^{\times}$ by reducing modulo $\p$ in $K$ and considering the corresponding action on $k^{\times} = \left(\mO_K/\p\right)^{\times}$.  In particular, for the extension $K_l$, there is a commutative diagram\\
\begin{center}
$ \xymatrix{   \prod_v G'_{l,v} \ar[d] \ar[r] & \prod_{\p} \Hom(\left(\mO_K/\p\right)^{\times} , \mu_l)  \ar[d]   \\ 
G'_l \ar[r]^{\cong}  & \Hom(\mO_K^{\times}, \mu_l)}$ \\
\end{center}
where the vertical arrow denotes reduction modulo $\p$, and the image is exactly the Frobenius element at $\p$.  Since $G$ permutes the Frobenius elements transitively, we deduce that the $G$-module generated by any Frobenius element is precisely the image of the map 
$$\Hom(\left(\mO_K/p\mO_K\right)^{\times} , \mu_l) \rightarrow \Hom(\mO_K^{\times}, \mu_l).$$
If $l \neq 2$, then $ \Hom(\mO_K^{\times}, \mu_l)$ has order $l^{n-1}$, and the map is surjective if and only if the image of $\mO_K^{\times}$ in $(\mO_K/p\mO_K)^{\times}/(\mO_K/p\mO_K)^{\times l}$ has order $l^{n-1}$, or equivalently if and only if $\mO_K^{\times}$ surjects onto the trace zero part of this space.  If $l=2$, then $ \Hom(\mO_K^{\times}, \mu_l)$ has order $2^{n}$, and the reduction map $\mO_K^{\times} \rightarrow (\mO_K/p\mO_K)^{\times}/(\mO_K/p\mO_K)^{\times 2}$ must be an isomorphism.  Now apply Theorem \ref{relatepsiandbl}.\\

 A priori there are $[K(\zeta_l):\mathbb{Q}]$ primes $\Q$ of $K(\zeta_{l})$ lying above $p$ since $p$ is split in $K/\mathbb{Q}$ and $l|p-1$.  However, for two primes $\Q$ and $\Q'$ lying above $p$, $\sigma_{\Q}$ and $\sigma_{\Q'}$ generate subgroups in $G_{l}$ of the same size which are conjugate in $\Gal(K_l/\mathbb{Q})$, and so $\sigma_{\Q}$ generates $G_l$ as a $G$-module if and only if $\sigma_{\Q'}$ does.  Hence the choice of Frobenius element $\sigma_{\Q} \in G_l$ is independent of $\Q|p$ (and so also independent of choice of $\p | p$), and we may choose the Frobenius element of any $\Q$ lying over $p$. 
 
\end{proof}

Using Theorem \ref{powerof2} we also deduce the following.

\begin{thm}\label{reformulationpowerof2}
Assume that $K$ satisfies condition \textbf{(G1)}.  Let $p \in \mathbb{Z}$ be a completely split prime in $K$.  The index of $\psi_p(\mO_K^{\times})$ in $\left( \mO_K/p\mO_K\right)^{\times}$ is a power of 2 times $ (p-1)/2$ if and only if for every $l|p-1$, $l \neq 2$, the Frobenius element $\sigma_{\Q}$ of any prime $\Q$ of $K(\zeta_l)$ lying above $p$ generates $G_l$ as a $G$-module.  

\end{thm}

\section{Multiquadratic fields}\label{mq}

We will now add the assumption that $K$ is a multiquadratic field.  Throughout this section and the next, $K$ will denote a totally real multiquadratic field containing a unit of norm $-1$ of degree $n=2^m$ over $\mathbb{Q}$ and $G=\Gal(K/\mathbb{Q})\cong (\mathbb{Z}/2\mathbb{Z})^{m}$, $m \geq 1$.    We also add the assumption that the ramification index of $2$ in $K/\mathbb{Q}$ satisfies $e_2(K/\mathbb{Q}) \leq 2$.  This assumption is only used in Lemma \ref{genusismq} to prove that the genus field of $K$ is also multiquadratic.  In this section we will reinterpret the equivalent conditions for $K$ to have minimal index mod $p$ given in Theorem \ref{reformulation} for multiquadratic fields.  We also prove that $K$ satisfies condition \textbf{(G1)}.  

\subsection{Determination of the map $\psi$ at $l=2$}

  In this section we prove that if $K$ has minimal index mod $p$, then $p \equiv 3 \mod 4$, and in this case the map $\phi_2$ is an epimorphism.  The following lemma uses in an essential way the fact that $K$ is a multiquadratic field.

\begin{lem}\label{lemma:sum-1generates}
An element $x=\sum_{i=1}^n x_i [g_i] \in \mathbb{F}_2[G]$ generates $\mathbb{F}_2[G]$ as a $G-$module if and only if $$\sum_{i=1}^n x_i \equiv 1 \mod 2.$$ 
\end{lem}

\begin{proof}
 Let $x= \sum_{i=1}^n x_i [g_i]$.  First suppose $\sum x_i \equiv 1 \mod 2$.  Then there are an odd number of nonzero coefficients of $x$.  Let $g_{i_1}, \ldots, g_{i_{2k+1}}$ be the $2k+1$ elements of $G$ with nonzero $x_{i_j}$ coefficients.  Then 
\begin{eqnarray*} 
x \cdot x &=& \sum_{j=1}^{2k+1} [g_{i_j}] \cdot \sum_{t=1}^{2k+1}[g_{i_t}] = \sum_{j=1}^{2k+1} \sum_{t=1}^{2k+1} [g_{i_j} \cdot g_{i_t}]\\
&=&  (2k+1)[g_1] + 2 \sum \sum_{j \neq t} [g_{i_j} \cdot g_{i_t}] = (2k+1)[g_1] + 0 = [g_1], 
\end{eqnarray*}
 since we are working in characteristic 2 and every element of $G$ has order 2.  Hence $x^2 = [g_1]$.  The action of $G$ gives us all elements of $\mathbb{F}_2[G]$ with exactly one nonzero coefficient, and adding these together will give us all elements of $\mathbb{F}_2[G]$ with exactly $2$, $3$, $\ldots$ , and $n$ nonzero coefficients.  Hence $x$ will generate $\mathbb{F}_2[G]$ as a $G-$module.  Conversely, if $\sum x_i \equiv 0 \mod 2$, then for any element $y$ in $\mathbb{F}_2[G]$, the coefficients of $x \cdot y$ will 
 always sum to zero, and so $x$ cannot generate all of $\mathbb{F}_2[G]$ as a $G-$module. 
 \end{proof}

By Lemma \ref{z2iso}, if $K$ has minimal index mod $p$, then there is an isomorphism of $\mathbb{F}_2[G]-$modules $\mO_K^{\times}/(\mO_K^{\times})^2 \rightarrow I/2I$.  Also, by Lemma \ref{torat2}, $I/2I$ is not isomorphic to $\mathbb{F}_2[G]$ as $G-$modules.  

\begin{lem}\label{prop:ok=f2}
 $\mO_K^{\times}/(\mO_K^{\times})^2 \cong \mathbb{F}_2[G]$ as $G-$modules.
 \end{lem}
 \begin{proof}

Consider the composition $$\mO_K^{\times}  \rightarrow \prod_{v|\infty} \mathbb{R}^{\times}/(\mathbb{R}^{\times})^2 \stackrel{\cong}{\longrightarrow} \mathbb{F}_2[G] \stackrel{Tr}{\longrightarrow}\mathbb{F}_2,$$ 
$$u \mapsto (\sigma_v (u))_{v|\infty}=((-1)^{t_v})_{v|\infty}\rightarrow \sum_{v|\infty} t_v [\sigma_v]\rightarrow \sum_{v|\infty} t_v, $$  
where we naturally identify the $n$ real infinite places $v$ with elements $\sigma_v$ of $G=\Gal(K/\mathbb{Q})$, and $Tr$ is the trace map.  This map factors through $\mO_K^{\times}/(\mO_K^{\times})^2$, giving a homomorphism of
$G-$modules $\mO_K^{\times}/(\mO_K^{\times})^2 \rightarrow \mathbb{F}_2[G]$.  By
assumption  there is a unit $\epsilon^{'}$ of $K$ of norm $-1$, and so the
image under the above map of $\epsilon^{'}$  is $1 \in
\mathbb{F}_2$.  By Lemma \ref{lemma:sum-1generates}, the
image of $\epsilon^{'}$ in $\mathbb{F}_2[G]$ generates
$\mathbb{F}_2[G]$.  This gives an epimorphism $\mO_K^{\times}/(\mO_K^{\times})^2
\rightarrow \mathbb{F}_2[G]$ onto a free module, which is an isomorphism since both
$\mO_K^{\times}/(\mO_K^{\times})^2 $ and $\mathbb{F}_2[G]$
have the same size.
\end{proof}

Thus if $p \equiv 1 \mod 4$, since $I/2I$ is not isomorphic to $\mathbb{F}_2[G]$, the map $\mO_K^{\times}/(\mO_K^{\times})^2 \rightarrow I/2I$ cannot be an isomorphism of $G-$modules.  By Lemma \ref{z2iso}, $K$ does not have minimal index mod $p$.   Next we turn to the case $p \equiv 3 \mod 4$.

\begin{lem}\label{3mod4true}
If $p \equiv 3 \mod 4$, then the map $\phi_2$ is an epimorphism.
\end{lem}
\begin{proof}
The image under $\psi$ of $\epsilon^{'}$ in $\prod_{i=1}^n (\mO_K/\p_i)^{\times}$ can be written 
$$(\epsilon^{'} \mod \p_1, \ldots , \epsilon^{'} \mod \p_n)= ((g_1(a))^{t_1}, \ldots , (g_n(a))^{t_n}),$$ 
where $a$ is a generator for $(\mO_K/\p)^{\times}$, and so $\prod_{i=1}^n g_i^{-1}(\epsilon^{'}) = N^K_{\mathbb{Q}}(\epsilon^{'}) =-1 \equiv \prod_{i=1}^n a^{t_i}=a^{\sum t_i} \mod \p$.  Since $a$ is a generator for $(\mO_K/\p)^{\times}$, $-1 \equiv a^{(p-1)/2} \mod \p$ and so $\displaystyle a^{(p-1)/2}\equiv a^{\sum_{i=1}^n t_i} \mod \p$.  Then $(p-1)/2\equiv \sum_{i=1}^n t_i \mod p-1$, implying $\sum_{i=1}^n t_i \equiv 1 \mod 2$ since $p \equiv 3 \mod 4$.  Thus 
$$\phi_2(\epsilon^{'})=\sum_{i=1}^n t_i [g_i] \in \{ x=\sum x_ig_i \in \mathbb{F}_2[G] | \sum x_i = 1 \}.$$ 
 By Lemma \ref{lemma:sum-1generates}, we see that  $\phi_2(\epsilon^{'})$ will generate $\mathbb{F}_2[G]$ as a $G-$module, and so $\phi_2$ is an epimorphism.
\end{proof}

We deduce that if $p \equiv 3 \mod 4$ it is always the case that the Frobenius element $\sigma_{\Q}$  of any prime $\Q$ of $K(\zeta_l)$ lying above $p$ generates $G_{2}=\Gal(K_{2}/K(\zeta_{2}))$ as a $G-$module.  We further conclude that the reformulation in Theorem \ref{reformulation} remains valid if we replace $K_2=K\left(\sqrt{\mO_K^{\times}}\right)$ with the simpler field $K(\sqrt{-1})$, since $p \equiv 3 \mod 4$ if and only if a prime ideal above $p$ generates $\Gal(K(\sqrt{-1})/K(\zeta_2))= \Gal(K(\sqrt{-1}/K))$ as a $G-$module. \\

Finally we prove that the condition $\mathbf{(G1)}$ is satisfied for every multiquadratic field $K$ containing a unit of norm $-1$.

\begin{thm}\label{G1satisfied}
If $\epsilon^{'}$ is a unit of norm $-1$, then $\gcd\left(\left[\mO_K^{\times}: \mathbb{Z}[G]\cdot \epsilon^{'} \right],2\right)=1.$
\end{thm}

\begin{proof}
 Let $M= \mathbb{Z}[G]\cdot \epsilon^{'}$ and $N=\mO_K^{\times}$; then the sequence   
 $$ 0 \longrightarrow M \stackrel{i}{\longrightarrow} N \longrightarrow N/M \longrightarrow 0$$
 is exact.  Using Lemma  \ref{prop:ok=f2} we deduce that the map $i$ is an isomorphism mod $2$.  By Nakayama's lemma, this implies that $\left(N/M\right)_{(2)}=0$.  Thus $N/M$ is torsion as a $\mathbb{Z}-$module and has order coprime to $2$.   
\end{proof}

 Combining this with the previous results of this section, we can now restate Theorem \ref{reformulation} for multiquadratic fields.\\

\begin{thm}\label{mqreformulation}  Let $K_l = K\left(\zeta_l, \sqrt[l]{\mO_K^{\times}}\right)$ if $l \neq 2$ is prime and let $K_2=K(\sqrt{-1})$.
For a prime $p \in \mathbb{Z}$ completely split in $K$, $K$ has minimal index mod $p$ if and only if for every $l|p-1$, the Frobenius element $\sigma_{\Q}$  of any prime $\Q$ of $K(\zeta_l)$ lying above $p$ generates $G_{l}=\Gal(K_{l}/K(\zeta_{l}))$ as a $G-$module.  
\end{thm}

\subsection{An equivalent condition on Frobenius elements for $l \neq 2$}

Fix a prime $l \neq 2$.  We will further reformulate the condition on the Frobenius elements $\sigma_{\Q}$ given in Theorem \ref{reformulation}.  Let $Q_i$, $ 1 \leq i \leq n-1$ denote the $n-1$ quadratic subfields of $K$.  For each $i$, let $H_i$ denote the index two subgroup of $G$ corresponding to the fixed field $Q_i$.  Then $H_i = \Gal(K/Q_i)$ and $Q_i=K^{H_i}$.  Since $l \nmid |G|=2^m$, $G_{l} \cong \oplus_{i=1}^{n-1} V_{i} $ as $G-$modules, where each $V_{i}$ is a one dimensional irreducible $G$-module.  Explicitly, since the action of an element $\sigma \in G_l$ on $K_{l}=K(\zeta_{l}, \sqrt[l]{\epsilon_{1}}, \ldots , \sqrt[l]{\epsilon_{n}})$
 is given by $\sigma(\sqrt[l]{\epsilon_k})= \zeta_{l}^{a_{k}} \sqrt[l]{\epsilon_k}$, for each $i$, $V_{i}= \mathbb{F}_{l} \left[ \sum_{h \in H_i} h - \sum_{g \notin H_i} g \right]$.  Let $\pi_{i} : G_{l} \rightarrow V_{i}$ denote the natural projection map onto the irreducible submodule $V_{i}$.  This map has kernel isomorphic to $\displaystyle  M_i:=\oplus_{j \neq i} V_{j}= \Gal(K_{l}/K_{l}^{M_i})$, and so $\sigma \in G_{l}$, $\pi_{i}(\sigma)=0$ if and only if $\sigma \in  \oplus_{j \neq i} V_{j}$, or equivalently $\sigma$ is the identity on the field $K_{l}^{M_i}$. Define $u_{i} := \prod_{h \in H_i}h(\epsilon) \in Q_i \subseteq K$ for $ 1 \leq i \leq n-1$,  where $\epsilon$ is the unit for which $K_{l}=K(\zeta_{l}, \sqrt[l]{\epsilon_{1}}, \ldots , \sqrt[l]{\epsilon_{n}})$ given by Lemma \ref{Washington}.\\

\begin{lem}
For every $1 \leq i \leq n-1$, $K_l^{M_i} =K(\zeta_{l}, \sqrt[l]{u_i})$.
\end{lem}

\begin{proof}
$[K_l^{M_i} : K(\zeta_l)] = l$, so the extension is cyclic generated by an $l^{th}$ root.  We show that $\sqrt[l]{u_i} \in K_l^{M_i}$. Consider for any $i \neq j$ the containment $ H_i \leq H_i H_j \leq G$.  The first containment is proper because $i \neq j$.  Since $[G:H_i]=2$, $H_i H_j = G$.  Then $\displaystyle | H_i\cap H_j|= \frac{|H_i||H_j|}{|H_i H_j|} = 2^{m-2}$, and so for every $i \neq j$, $[H_j: H_i \cap H_j] = 2$.  Then for any $\sigma \in V_j=\mathbb{F}_l[\sum_{h\in H_i} h- \sum_{g \notin H_i g} g ]$, say $\sigma = a(\sum_{h\in H_i} h- \sum_{g \notin H_i} g)$, we have
$$  \sigma(\sqrt[l]{u_i})=\sigma\left(\prod_{h\in H_i} \sqrt[l]{h(\epsilon)}\right) 
  =\prod_{h \in H_i \cap H_j} \zeta_l^a \sqrt[l]{h(\epsilon)} \cdot \prod_{h \in H_i \setminus H_j} \zeta_l^{-a} \sqrt[l]{h(\epsilon)}=  \sqrt[l]{u_i}.  $$
\end{proof}

For each $1 \leq i \leq n-1$, let $\qB_{i}:= \Q \cap \mO_{Q_i(\zeta_l)}$; then $\qB_{i}$ is a completely split prime of $Q_{i}(\zeta_{l})$  that $\Q$ lies above.  

\begin{lem}\label{lemma:quadsub}
$\sigma_{\Q}$ generates $G_{l}$ as a $G$-module if and only if for all $1 \leq i \leq n-1$, $(\qB_{i}, Q_{i}(\zeta_{l}, \sqrt[l]{u_{i}})/Q_{i}(\zeta_{l})) \neq (id)$.
\end{lem}

\begin{proof}
Since $G_{l} \cong  \oplus_{i=1}^{n-1} V_{i} $ as $G$-modules with $V_{i}$ irreducible $G$-modules, $\sigma_{\Q}$ generates $G_{l}$ as a $G$-module if and only if $\sigma_{\Q} |_{K_{l}^{M_i}} \neq (id)$ for all $1 \leq i \leq n-1$.  Now, 
\begin{eqnarray*} 
\sigma_{\Q} |_{K_{l}^{M_i}} &=& (\Q, K_{l}/K(\zeta_{l}))|_{K_{l}^{M_i}} \\
&=& (\Q, K_{l}^{M_i}/K(\zeta_{l})) \\
&=& (\Q, K(\zeta_{l}, \sqrt[l]{u_i})/K(\zeta_{l})).
\end{eqnarray*}
Thus $\sigma_{\Q}$ generates $G_{l}$ as a $G$-module if and only if for every $ 1 \leq i \leq n-1$, $(\Q, K(\zeta_{l}, \sqrt[l]{u})/K(\zeta_{l})) \neq (id)$.  The element $u_{i} = \prod_{h \in H_i}h(\epsilon) \in Q_i=K^{H_i} $ since $u_i$ is fixed by every element of $H_i$.  Then $(\Q, K(\zeta_{l}, \sqrt[l]{u_{i}})/K(\zeta_{l})) \neq (id)$ if and only if the Frobenius element $(\qB_{i}, Q_{i}(\zeta_{l}, \sqrt[l]{u_{i}})/Q_{i}(\zeta_{l})) \neq (id)$, since both elements have order a power of $l$ and their order differs by at most $2^{m-1}=[K:Q_{i}]$.  
\end{proof}

\section{Application of analytic number theory}\label{proofs}

\subsection{Classical argument}

Recall that $K$ is a totally real multiquadratic field containing a unit of norm $-1$.  We will adapt techniques originally given by Hooley \cite{H} in his proof of Artin's conjecture on primitive roots to our setting.  Lemma \ref{lemma:quadsub} allows us to convert conditions pertaining to the multiquadratic field $K$ into conditions pertaining to its quadratic subfields.  Several of the proofs, including those of Lemmas \ref{M3lemma} and \ref{M_1}, are then modifications of Roskam's results for quadratic fields (see \cite{R2}).  These two lemmas will not directly generalize to other number fields $K$.  We fix notation which we will use throughout the rest of this section.  \\

\begin{nota}
 Let $l$ denote a prime number.  Let $r$ and $k$ denote positive integers, with $k$ squarefree. 
 \end{nota}
 
In what follows we disregard the prime $p=2$ since we are only interested in the asymptotic behavior of $p$.  Recall $K_l=K\left(\zeta_l, \sqrt[l]{\mO_K^{\times}}\right)$ for $l \neq 2$ and $K_2 = K(\sqrt{-1})$.  For $p \in \mathbb{Z}$ a prime that splits completely in $K$, let $R(p,l)=1$ if $l|p-1$ and the Frobenius element $\sigma_{\Q}$ does not generate $G_{l}=\Gal(L_l/K(\zeta_l))$ as a $G-$module for any $\Q|p$ of $K(\zeta_l)$, and 0 otherwise.  Define the following for $x$, $\delta$, $\delta_{1}$, $\delta_{2}  \in \mathbb{R}$.
$$ N(x, \delta)= \left| \{ p \text{ split in } K/\mathbb{Q} : p \leq x, R(p, l)=0 \ \forall  \ l \leq \delta \} \right|$$

$$ N(x)=N(x, x-1)= \left| \{ p \text{ split in } K/\mathbb{Q} : p \leq x, R(p, l)=0 \ \forall  \ l \leq x-1 \} \right|$$

$$ P(x, k)= \left| \{ p \text{ split in } K/\mathbb{Q} : p \leq x, R(p, l)=1 \ \forall  \ l|k \} \right|$$

$$ M(x, \delta_{1}, \delta_{2})= \left| \{ p \text{ split in } K/\mathbb{Q} : p \leq x, R(p, l)=1  \ \exists  \ l \in (\delta_{1}, \delta_{2}\right]   \} | $$

$$ S = \{ \text{primes } p \text{ completely split in } K/\mathbb{Q} :\forall l| p-1, \sigma_{\Q} \text{ generates } G_{l} \text{ as a } G- \text{module } \text{ for any } \Q|p \text{ of }K(\zeta_l) \} $$

By Theorem \ref{mqreformulation}, $K$ has minimal index mod $p$ if and only if $p$ is counted in $S$, and so we wish to show that the set $S$ has a positive density
 $$\delta(S) =\lim_{x \rightarrow \infty} \frac{  \left| \{ p \in S : p \leq x \} \right|}{\pi (x)}= \lim_{x \rightarrow \infty} \frac{N(x)}{\pi (x)}$$
 in the set of all primes, where $\pi(x)$ denotes the number of prime numbers $p$ less than or equal to $x$.   We will do so by investigating the limit as $ x \rightarrow \infty$ of  the ratio $\displaystyle \frac{N(x)}{x/\log x}$.  \\
 
 Define $M= 2n+3$, where $n=[K:\mathbb{Q}]$; note that $M$ only depends on the degree of the multiquadratic field $K$.  Let 
$$\xi_1(x)= \frac{1}{2M} \log{(x)}, \text{ } \displaystyle \xi_2(x)= \frac{\sqrt{x}}{\log^2{x}}, \text{ and } \xi_3(x)= \sqrt{x}\log{x}.$$  
In addition, define $M_1(x) := M(x, \xi_1(x), \xi_2(x)) $, $ M_2(x):= M(x, \xi_2(x), \xi_3(x))$, and $M_{3}(x) :=   M(x, \xi_{3}(x), x-1) $.  As in the argument of Hooley, we deduce
$$\displaystyle N(x)= N(x, \xi_{1}(x)) + O(M(x, \xi_1(x), x-1)).$$
Also, we trivially have 
\begin{eqnarray*}
M(x, \xi_{1}(x), x-1) &\leq& M(x, \xi_{1}(x), \xi_{2}(x)) +  M(x, \xi_{2}(x), \xi_{3}(x)) +  M(x, \xi_{3}(x), x-1)\\
&=& M_1(x) + M_2(x) + M_3(x).\\
\end{eqnarray*}

\begin{lem}\label{M3lemma}
$\displaystyle M_{3}(x)  = O \left( \frac{x}{\log^{2}x}\right).$
\end{lem}

\begin{proof}
If $p$ is counted in $M_{3}(x)$, then for some $l$, $ x-1 \geq l \geq \xi_{3}(x)$, $l|p-1$ and for a prime $\Q|p$ of $K(\zeta_l)$, $\sigma_{\Q}$ does not generate $G_{l}$ as a $G$-module.  Since we are interested in the asymptotic behavior of $x$, we may assume that $\xi_3(x)$ is not small, and so $l \neq 2$.  By Lemma \ref{lemma:quadsub}, this implies that for some $1 \leq i \leq n-1$, $(\qB_{i}, Q_{i}(\zeta_{l}, \sqrt[l]{u_{i}})/Q_{i}(\zeta_{l}))=(id)$, so $\qB_{i}=\Q \cap \mO_{Q_i(\zeta_l)}$ is a prime of $Q_i(\zeta_l)$ that splits completely in $Q_{i}(\zeta_{l}, \sqrt[l]{u_{i}})$, where $\displaystyle u_{i} = \prod_{h \in H_i}h(\epsilon) \in Q_i$.  Hence $u_i^{(p-1)/l} - 1 \in \qB_i$; since also $u_i^{(p-1)/l} - 1 \in \mO_{Q_i}$, $u_i^{(p-1)/l} - 1 \in \qB_i \cap \mO_{Q_i}=:\q_i$ for $\q_i$ a prime of $Q_{i}$.  Thus $\q_i | u_{i}^{(p-1)/l} - 1$, and so $N_{\mathbb{Q}}^{Q_i}(\q_i)=p \leq N_{\mathbb{Q}}^{Q_i}\left(u_{i}^{(p-1)/l}\right)$.  Now, since $ x-1 \geq l \geq \xi_{3}(x)=\sqrt{x} \log x$ and $p\leq x$, we have 

$$ \frac{1}{l} \leq \frac{1}{\sqrt{x} \log x}, \text{ and so } \frac{p-1}{l} \leq \frac{p-1}{\sqrt{x}\log x} \leq \frac{x}{\sqrt{x}\log x} = \frac{\sqrt{x}}{\log x}.$$  Thus every such $p$ counted in $M_{3}(x)$ divides $\displaystyle \prod_{r<\frac{\sqrt{x}}{\log x}} \prod_{i=1}^{n-1} N_{\mathbb{Q}}^{K_i}(u_{i}^{r} - 1)$, where $r \in \mathbb{N}$.  Each $p$ is at least 2, so $\displaystyle 2^{ M_{3}(x)} \leq \prod_{r<\frac{\sqrt{x}}{\log x}} \prod_{i=1}^{n-1} |N_{\mathbb{Q}}^{Q_i}(u_{i}^{r} - 1)|$.  Let $A_{1} > \max_{i, g_{j}} |g_{j}(u_{i})|$; notice that since there are only finitely many distinct $\epsilon$ for varying $l$, there are only finitely many distinct $u_i$ and so $A_1$ is finite.  For each $i$, $\displaystyle |N_{\mathbb{Q}}^{Q_i}(u_{i}^{r} - 1)| < (A_{1}^{r} + 1)^{2}$.  Thus $$ \prod_{i=1}^{n-1} |N_{\mathbb{Q}}^{Q_i}(u_{i}^{r} - 1)| < (A_{1}^{r} + 1)^{2(n-1)} < A_{2}^{2r},$$
 where $A_{2} > (A_{1} + 1)^{n-1}$, and $A_{1}, A_{2} =O(1)$.  Hence
$$  2^{ M_{3}(x)} \leq \prod_{r<\frac{\sqrt{x}}{\log x}} \prod_{i=1}^{n-1} |N_{\mathbb{Q}}^{Q_i}(u_{i}^{r} - 1)| < \prod_{r<\frac{\sqrt{x}}{\log x}} A_{2}^{2r}.$$ 

Taking logs, we have 
\begin{eqnarray*}
  M_{3}(x) &\leq& \sum_{r<\frac{\sqrt{x}}{\log x}} \frac{\log(A_{2}^{2r})}{\log 2} = \sum_{r<\frac{\sqrt{x}}{\log x}} \frac{2r\log(A_{2})}{\log 2}\\ 
&=& \frac{2\log(A_{2})}{\log 2}\sum_{r<\frac{\sqrt{x}}{\log x}} r 
 \leq \frac{2\log(A_{2})}{\log 2}\sum_{r<\frac{\sqrt{x}}{\log x}} \frac{\sqrt{x}}{\log x}\\
 &\leq& \frac{2\log(A_{2})}{\log 2} \frac{\sqrt{x}}{\log x} \frac{\sqrt{x}}{\log x} =  \frac{2\log(A_{2})}{\log 2} \frac{x}{\log^{2}(x)}.
\end{eqnarray*}
Hence $\displaystyle M_3(x)  = O\left( \frac{x}{\log^{2}x}\right)$.
\end{proof}

\begin{lem}
$\displaystyle M_2(x) = O \left( \frac{x \log\log x}{\log^2 x}\right)$.
\end{lem}

\begin{proof}
If $p$ is counted in $M_2(x)$, then there is a prime $l_0 \in (\xi_2(x), \xi_3(x)]$ such that $R(p, l_0) =1$, and so $p$ will be counted in $P(x, l_0)$.  Thus $$ M_2(x) \leq \sum_{l \in (\xi_2(x), \xi_3(x)]} P(x, l).$$
  Now, for any $l$, $P(x,l) \leq   \left| \{ p \leq x : p \equiv 1 \mod l \} \right|$, since the primes counted in $R(p, l)$ satisfy additional conditions.  Now, for each $l$ prime, the Brun-Titchmarsh Theorem states that 
$$\left| \{ p \leq x \text{ prime }: p \equiv 1 \mod l \} \right| \leq \frac{bx}{\phi(l) \log(\frac{x}{l})} =  \frac{bx}{(l-1) \log(\frac{x}{l})}$$
for some fixed absolute constant $b$, and so we deduce that $\displaystyle M_2(x) \leq \sum_{l \in (\xi_2(x), \xi_3(x)]} \frac{bx}{(l-1) \log(\frac{x}{l})}$.  \\

Now notice that $\displaystyle \frac{x}{(l-1)\log(\frac{x}{l})} \leq  \frac{8x}{l \log x}$ for all $l \in (\xi_2(x), \xi_3(x)]$.  To see this, note that since \\ $l \leq \xi_3(x) = \sqrt{x}\log x $,  $\displaystyle x/l \geq \sqrt{x}/\log x $, and so $  \left(x/l\right)^4 \geq x^2/\log^4 x \geq x.$ Thus 
$$ \log x \leq \log \left(\left(x/l\right)^4\right) = 4 \log \left(x/l\right) \text{, or equivalently }  \frac{1}{\log(x/l)} \leq \frac{4}{\log x} \text{.  Then }  \frac{x}{(l-1)\log(\frac{x}{l})} \leq \frac{8x}{l \log x},$$
 and so we find that 
$$M_2(x) \leq  \sum_{l \in (\xi_2(x), \xi_3(x)]} \frac{bx}{(l-1) \log(\frac{x}{l})} \leq \sum_{l \in (\xi_2(x), \xi_3(x)]} \frac{8bx}{l\log x},$$ and so $\displaystyle M_2(x) = O \left( \frac{x}{\log x} \sum_{l \in (\xi_2(x), \xi_3(x)]} \frac{1}{l}\right) $.\\

Next we claim that 
$$ \frac{x}{\log x} \sum_{l \in (\xi_2(x), \xi_3(x)]} \frac{1}{l} = O\left(  \frac{x}{\log^2 x} \sum_{l \in (\xi_2(x), \xi_3(x)]} \frac{\log l}{l}\right) .$$
  Since $\displaystyle l \geq \sqrt{x}/\log^2 x$ for every $l \in (\xi_2(x), \xi_3(x)]$, $\displaystyle l^4 \geq x^2/\log^8 x \geq x$, and so $\log x \leq \log (l^4) = 4 \log l$ or equivalently $\displaystyle 1 \leq 4 \log l/\log x$.  Hence 
\begin{eqnarray*}
\frac{x}{\log x} \sum_{l \in (\xi_2(x), \xi_3(x)]} \frac{1}{l} &\leq& \frac{4x}{\log x} \sum_{l \in (\xi_2(x), \xi_3(x)]} \frac{\log l}{l \log x}\\
 &=& \frac{4x}{\log^2 x} \sum_{l \in (\xi_2(x), \xi_3(x)]} \frac{\log l}{l},\\
\end{eqnarray*}
and so $\displaystyle \frac{x}{\log x} \sum_{l \in (\xi_2(x), \xi_3(x)]} \frac{1}{l} = O\left(  \frac{x}{\log^2 x} \sum_{l \in (\xi_2(x), \xi_3(x)]} \frac{\log l}{l}\right) $.  Thus 
$$M_2(x) = O\left( \frac{x}{\log^2 x} \sum_{l \in (\xi_2(x), \xi_3(x)]} \frac{\log l}{l}\right).$$

We now can use a theorem of Mertens to bound the last sum.  Define 
$$ S(x):=\sum_{l \in (\xi_2(x), \xi_3(x)]} \frac{\log l}{l}.$$  
Mertens theorem states that $\displaystyle \log x - \sum_{l \leq x \text{ prime}} \frac{\log l}{l} = O(1)$ as $ x \rightarrow \infty$.  Thus there exists a constant $t >0$ (independent of $x$) such that $\displaystyle \left|\log(\xi_3(x)) - S(x) - \sum_{l \leq \xi_2(x)} \frac{\log l}{l}\right| \leq t$ and $\displaystyle \left| \log(\xi_2(x)) - \sum_{l \leq \xi_2(x)} \frac{\log l}{l}\right| \leq t$.  Hence $$\left|\log(\xi_3(x)) -\log(\xi_2(x)) - S(x)\right|  =\left|(\log(\xi_3(x)) - S(x) - \sum_{l \leq \xi_2(x)} \frac{\log l}{l}) -(\log(\xi_2(x)) - \sum_{l \leq \xi_2(x)} \frac{\log l}{l})\right|$$ $$\leq \left|\log(\xi_3(x)) - S(x) - \sum_{l \leq \xi_2(x)} \frac{\log l}{l}\right| + \left| \log(\xi_2(x)) - \sum_{l \leq \xi_2(x)} \frac{\log l}{l}\right| \leq 2t$$\\
by the triangle inequality.  Thus $\displaystyle S(x)\leq \log(\xi_3(x))- \log (\xi_2(x)) + 2t= \log\left(\frac{\xi_3(x)}{\xi_2(x)}\right) + 2t$.  Now, $$ \frac{\xi_3(x)}{\xi_2(x)} = (\sqrt{x}\log x)\cdot \frac{\log^2 x}{\sqrt{x}}=\log^3 x,$$
 and so $\displaystyle \log\left(\frac{\xi_3(x)}{\xi_2(x)}\right)= \log \log^3 x = 3 \log\log x$.  This implies $$ S(x) \leq  \log\left(\frac{\xi_3(x)}{\xi_2(x)}\right) + 2t = 3 \log \log x + 2t \leq h \log \log x$$ for some absolute constant $h$.  Thus we can conclude that 
$$ M_2(x)= O\left(  \frac{x}{\log^2 x} S(x)\right)  =O\left(\frac{x \log \log x}{\log^2 x} \right)\text{, as claimed.}$$ 
\end{proof}

Because $\displaystyle \frac{x}{\log^2 x} =O\left( \frac{x\log\log x}{\log^2 x}\right)$, we have

\begin{eqnarray*}
 N(x)&=& N(x, \xi_{1}(x)) + O(M(x, \xi_{1}(x), x-1)) \\
 	&=& N(x, \xi_1(x)) + O(M_1(x)) + O\left(\frac{x \log \log x}{\log^2 x} \right). \hspace{200pt} \mathbf{(1)}
\end{eqnarray*} 

\subsection{Completion of proof using the generalized Riemann hypothesis}

From $\mathbf{(1)}$ we see that to estimate $N(x)$ it remains to estimate both $N(x, \xi_1(x))$ and $M_1(x)$.  As with $M_2(x)$, we have $\displaystyle M_1(x) \leq \sum_{l \in (\xi_1(x), \xi_2(x)]} P(x, l)$.  Also, by the inclusion/exclusion principle we have 
$$\displaystyle N(x, \xi_1(x)) = \sum_{k \in U(x)} \mu(k)P(x,k),$$ 
where $\mu(x)$ denotes the M\"obius function and $U(x)= \left\{ k : \text{if } l| k \text{ then } l \leq \xi_1(x) \right\}$.  For $k$ a squarefree integer, define 

 $$K_k= \prod_{l|k} K_l  = \left \{ \begin{array}{ll}
                                K(\zeta_k, \sqrt[k]{\mO_K^{\times}})  & \text{ if }2 \nmid k \\
                                 K(\sqrt{-1}, \zeta_b, \sqrt[b]{\mO_K^{\times}}) & \text{ if }k=2b    \end{array}
\right.  $$
to be the compositum of
the fields $K_l$ for $l|k$.  Then $K_k$ is a Galois extension of $\mathbb{Q}$, and for all $k$, $\displaystyle k^{n-1}  \phi(k)/2\leq \left[K_k: \mathbb{Q} \right]\leq 2n \phi(k)k^{n-1} $, where $n = [K:\mathbb{Q}]$ and $\phi(k)$ denotes the Euler phi function.  Define $d_k := |disc(K_k/\mathbb{Q})|$.  The following technical lemma will be useful for our estimate of $P(x,k)$.  The proof can be found in Roskam \cite{R2}.

\begin{lem}\label{tech}\cite[Prop.13]{R2}
There is a constant $\kappa_1$, depending on $\epsilon$,  such that for all $k >1$
and all subfields $F \subseteq K_k$ we have $\log d_F^{1/[F:\mathbb{Q}]}
\leq \kappa_1
\log k$, where $d_{F/\mathbb{Q}}$ is the absolute value of the discriminant of $F/\mathbb{Q}$.
\end{lem}

To estimate the sum $P(x,k)$ we will use the following conditional result
originally due to Lagarias-Odlyzko \cite{LO} and found in \cite{R2}.
\begin{thm}\label{Serre}
Let $F/\mathbb{Q}$ be a Galois extension with Galois group $G$, let $C$ be a
union of conjugacy classes of $G$, and denote the absolute value of the
discriminant of $F$ by $d_F$.  Define
$$\pi_C (x) = |\left\{p \leq x \text{ unramified in } F/\mathbb{Q}:(p,
F/\mathbb{Q}) \subseteq C\right\}|.$$
Then the generalized Riemann hypothesis (abbreviated GRH) implies that there is an absolute
constant $\kappa_2$ such that for all $x \geq 2$ the following inequality holds:
$$|\pi_C(x) - \frac{|C|}{|G|}Li(x)| \leq \kappa_2 |C| \sqrt{x}(\log
d_F^{1/[F:\mathbb{Q}]} + \log x),$$
with $Li(x) = \int_2^{\infty} \frac{dt}{\log t}$ the logarithmic integral.
\end{thm}

This theorem is applied to two different sets of fields to finish the proof.  First it is applied in the proof of the following. 

\begin{lem}\label{M_1}
$\displaystyle M_1(x) = O\left(\frac{x}{\log^2 x}\right)$.
\end{lem}

\begin{proof}
We have shown that $\displaystyle M_1(x) \leq \sum_{ l \in (\xi_1(x), \xi_2(x)]} P(x,l)$.  Again since $l  > \xi_1(x)$ we may assume that $l \neq 2$.  A prime $p$ contributes to the sum $P(x,l)$ if and only if $p$ is a prime split in $K$ with $p \leq x$ and $R(p,l)=1$; that is, $l|p-1$ and $\sigma_{\Q}$ does not generate $G_l$ as a $G$-module for any $\Q|p$.   By Lemma ~\ref{lemma:quadsub} there is an $i$, $1 \leq i \leq n-1$ such that $(\qB_i, Q_i(\zeta_l, \sqrt[l]{u_i})/Q_i(\zeta_l)) = (id)$, where $\qB_i$ is a prime of $Q_i(\zeta_l)$ lying above $p$.  Notice $R(p,l)=1$ if and only if $p$ splits completely in the extension $F_{l_{i}}=K(\zeta_l,\sqrt[l]{u_{i}})/\mathbb{Q}$ for some $1 \leq i \leq n-1$.  Hence a prime $p$ contributes to the sum $P(x,l)$ if and only if $p \leq x$ is split completely in $F_{l_i}/\mathbb{Q}$ for some $i$, $1 \leq i \leq n-1$.  We deduce
\begin{eqnarray*}
M_1(x) &\leq& \sum_{ l \in (\xi_1(x), \xi_2(x)]} P(x,l)\\
&\leq& \sum_{ l \in (\xi_1(x), \xi_2(x)]} |\left\{p \leq x: p \text{ splits in } F_{l_i}/\mathbb{Q} \text{ }\exists\text{ } 1 \leq i \leq n-1 \right\}|\\
&\leq&  \sum_{ l \in (\xi_1(x), \xi_2(x)]} \sum_{i=1}^{n-1} |\left\{p \leq x: p \text{ splits in } F_{l_i}/\mathbb{Q} \right\}|.\\
\end{eqnarray*}

Because the field extensions $F_{l_i}/\mathbb{Q}$ are Galois, we can apply Theorem ~\ref{Serre} to these fields $F_{l_i}/\mathbb{Q}$ with $C=C_{l_i}$ the completely split conjugacy class in order to get an estimate on $|\left\{p \leq x: p \text{ splits in } F_{l_i}/\mathbb{Q} \right\}|$.  For any $l$ and any $ 1 \leq i \leq n-1$, let $C=C_{l_i}= (id)$; then $\pi_C(x)= |\left\{p \leq x: (p,  F_{l_i}/\mathbb{Q})= (id) \right\}|=  |\left\{p \leq x: p \text{ splits in } F_{l_i}/\mathbb{Q} \right\}|$.  By Theorem ~\ref{Serre}, the GRH implies that there is an absolute constant $\kappa_2$ such that for all $x \geq 2$, $$\left|\pi_C(x) - \frac{1}{[F_{l_i}:\mathbb{Q}]}Li(x)\right| \leq \kappa_2 \sqrt{x}\left(\log
d_{F_{l_i}}^{1/[F_{l_i}:\mathbb{Q}]} + \log x\right),$$
since $|C_{l_i}|= 1$.  In addition, since $F_{l_i} \subseteq K_l$ for every $i$ and $l$, by Lemma ~\ref{tech}, $d_{F_{l_i}}^{1/[F_{l_i}:\mathbb{Q}]} \leq \kappa_1 \log l$, and so  $\displaystyle |\pi_C(x) - \frac{1}{[F_{l_i}:\mathbb{Q}]}Li(x)| = O\left( \sqrt{x}\log(lx)\right)$ for every $l$ and $i$.  Thus 
\begin{eqnarray*}
M_1(x) &\leq&  \sum_{ l \in (\xi_1(x), \xi_2(x)]} \sum_{i=1}^{n-1} |\left\{p \leq x: p \text{ splits in } F_{l_i}/\mathbb{Q} \right\}|\\
&\leq&  \sum_{ l \in (\xi_1(x), \xi_2(x)]} \sum_{i=1}^{n-1} \left(\frac{1}{[F_{l_i}:\mathbb{Q}]} Li(x) + O\left(\sqrt{x} \log(lx)\right)\right).\\
\end{eqnarray*}

Now, $[F_{l_i}:\mathbb{Q}] = n l (l-1)$ for all but finitely many $l$, and we may assume $\xi_1(x)>l$ for the finite number of these $l$. Since $nl(l-1) \geq  nl^2/2$ for all $l$, we deduce $1/[F_{l_i}:\mathbb{Q}] \leq 2/nl^2$ for all $l>\xi_1(x)$.  Hence 
\begin{eqnarray*}
M_1(x) &\leq&  \sum_{ l \in (\xi_1(x), \xi_2(x)]} \sum_{i=1}^{n-1} \left(\frac{1}{[F_{l_i}:\mathbb{Q}]} Li(x) + O\left(\sqrt{x} \log(lx)\right)\right)\\
&\leq&  \sum_{ l \in (\xi_1(x), \xi_2(x)]} \sum_{i=1}^{n-1} \frac{2}{nl^2}Li(x) +  \sum_{ l \in (\xi_1(x), \xi_2(x)]} \sum_{i=1}^{n-1} O(\sqrt{x}\log(lx))\\
&=& Li(x)(n-1) \sum_{l>\xi_1(x)}\frac{2}{nl^2} + (n-1) \sum_{ l \in (\xi_1(x), \xi_2(x)]}  O\left(\sqrt{x}\log(lx)\right)\\
&=& Li(x)(n-1) \sum_{l>\xi_1(x)}\frac{2}{nl^2}  + O\left(\sqrt{x}\log x \sum_{l \leq \xi_2(x)} 1\right).
\end{eqnarray*}

The former sum is bounded above by $\displaystyle \int_{\xi_1(x)}^{\infty} \frac{2}{nx^2} dx=\frac{2}{n\xi_1(x)}$, and the latter sum is bounded by the number of primes less than $\xi_2(x)$. By the Prime Number Theorem, the number of primes less than or equal to $z$ is $\displaystyle O\left( \frac{z}{\log z}\right)$, and $Li(x)$ is $\displaystyle O\left(\frac{x}{\log x}\right)$.  Thus 
\begin{eqnarray*}
M_1(x) &=& O\left( \frac{x}{\log x} \cdot \frac{1}{\xi_1(x)}\right) + O\left(\sqrt{x}\log x \frac{\xi_2(x)}{\log(\xi_2(x))}\right) \\
&=& O\left(\frac{x}{\log^2(x)}\right) + O\left(\frac{x}{\log^2 x} \cdot \frac{\log x}{\log{\xi_2(x)}}  \right),
\end{eqnarray*}
using the definitions of $\xi_1(x)$ and $\xi_2(x)$.  Lastly, since $ \xi_2(x)^4 =x^2/\log^8x \geq x$ for $x$ large enough, we see that $\log x \leq \log(\xi_2(x)^4)=4\log(\xi_2(x))$, and so 
$$O\left( \frac{x}{\log^2 x} \cdot \frac{\log x}{\log \xi_2(x)} \right) = O \left( \frac{x}{\log^2 x} \cdot 1\right).$$
Hence $\displaystyle M_1(x)=O\left(\frac{x}{\log^2(x)}\right) + O\left(\frac{x}{\log^2(x)}\right) =O\left(\frac{x}{\log^2(x)}\right) $.
\end{proof}

Next we apply Theorem ~\ref{Serre} to the fields $F=K_k$ to estimate the sum $P(x,k)$.  Let $C_k$ denote the union of conjugacy classes in $\Gal(K_k
/\mathbb{Q})$,
$$C_k= \bigcup \left\{\sigma \in \Gal(K_k/\mathbb{Q}): \sigma|_{K}=(id), \sigma|_{\mathbb{Q}(\zeta_l)}=(id) \text{ and }\right.$$
$$ \left. \sigma|_{K_l} \text{ does not generate } G_l \text{ as a }
G-\text{module for all }  l|k \right\}.$$
  Then by definition
$\pi_{C_k}(x)=P(x,k)$.  By Theorem ~\ref{Serre} the GRH
implies that $$|P(x,k) - \frac{|C_k|}{[K_k:\mathbb{Q}]}Li(x)| \leq \kappa_2
|C_k|\sqrt{x}(\log d_F^{1/[F:\mathbb{Q}]} + \log x).$$
This, combined with Lemma \ref{tech}, shows
\begin{eqnarray*}
|P(x,k) - \frac{|C_k|}{[K_k:\mathbb{Q}]}Li(x)| &\leq&  \kappa_2 |C_k|\sqrt{x}(\kappa_1
\log k + \log x)\\
&=& O\left( |C_k| \sqrt{x} \log (kx)\right)
\end{eqnarray*}
Hence  $ \displaystyle P(x,k) = \frac{|C_k|}{[K_k:\mathbb{Q}]}Li(x)+O\left( |C_k| \sqrt{x} \log (kx)\right).\hfill \mathbf{(2)} $\\

\begin{lem}\label{Nx1bound}
$\displaystyle N(x,\xi_1(x))= \sum_{k \in U(x)} \frac{\mu(k)|C_k|}{[K_k:\mathbb{Q}]}\frac{x}{\log x} + O\left( \frac{x}{\log^2 x} \right).$
\end{lem}

\begin{proof}

We have that $\displaystyle N(x, \xi_1(x)) = \sum_{k\in U(x)} \mu(k) P(x,k)$, where $U(x)= \left\{k: \text{if } l|k \text{ then } l \leq \xi_1(x) \right\}$.  Using equation $\mathbf{(2)}$ for $P(x,k)$ we deduce 
\begin{eqnarray*}
N(x, \xi_1(x)) &=& \sum_{k \in U(x)} \mu(k)\left[ \frac{|C_k|}{[K_k:\mathbb{Q}]}Li(x)+O\left( |C_k| \sqrt{x}\log x\right)\right]\\
&=& \sum_{k \in U(x)} \frac{\mu(k)|C_k|}{[K_k:\mathbb{Q}]}Li(x) + \sum_{k \in U(x)}O\left( |C_k| \sqrt{x}\log x\right).
\end{eqnarray*}
 Notice that by choice of $M$, $(1+n)/M = (1+n)/(2n+3) <1/2$.  Also, if $k \in U(x)$, then $k \leq x^{1/M}$.  For every $k$ we have the obvious inequality $|C_k| \leq [K_k:\mathbb{Q}] \leq 2n\phi(k)k^{n-1} \leq 2nk^n$, and so 
\begin{eqnarray*}
N(x, \xi_1(x)) - \sum_{k \in U(x)} \frac{\mu(k)|C_k|}{[K_k:\mathbb{Q}]}Li(x)&=& \sum_{k \in U(x)} O\left( |C_k| \sqrt{x} \log x\right)\\
&\leq&  \sum_{k \leq x^{1/M}} O\left( k^n \sqrt{x}\log x\right) \\
&=& O\left(x^{1/M} x^{n/M} \sqrt{x}\log (x^{1+1/M})\right) \\
&=& O\left(x^{\frac{4n+5}{4n+6}} \log (x)\right)=  O\left( \frac{x}{\log^2 x} \right).
\end{eqnarray*}

\end{proof}

Combining Lemmas \ref{M_1} and  \ref{Nx1bound} with the equation $\mathbf{(1)}$, we arrive at\\

$ \displaystyle N(x) =   \sum_{k \in U(x)} \frac{\mu(k)|C_k|}{[K_k:\mathbb{Q}]} \cdot Li(x) + O\left( \frac{x \log \log x}{\log^2 x} \right).  \hfill \mathbf{(3)}$\\

It remains to show that the sum $\displaystyle \sum_{k \in U(x)} \frac{\mu(k)|C_k|}{[K_k:\mathbb{Q}]}$ approaches a limit as 
$x \rightarrow \infty$. To do so, we will need several lemmas.

\begin{lem}\label{Krabelian}
The largest abelian extension of $\mathbb{Q}$ contained in $K_r$ is $K(\zeta_r)$.
\end{lem}
\begin{proof}
The sequence 
$$1 \rightarrow \Gal(K_r/K(\zeta_r)) \longrightarrow \Gal(K_r/\mathbb{Q}) \longrightarrow \Gal(K(\zeta_r)/\mathbb{Q}) \rightarrow 1 $$
given by restriction is exact, and so $\Gal(K_r/\mathbb{Q}) \cong \Gal(K(\zeta_r)/\mathbb{Q})  \rtimes \Gal(K_r/K(\zeta_r))$.  Since $r$ is odd, no extension of $K(\zeta_r)$ in $K_r$ is abelian.  Since $\Gal(K(\zeta_r)/\mathbb{Q})$ is abelian, we deduce
$$\Gal(K_r/\mathbb{Q})^{ab} \cong \Gal(K(\zeta_r)/\mathbb{Q}) \rtimes (id) \cong \Gal(K(\zeta_r)/\mathbb{Q}),$$
and so $K(\zeta_r)$ is the largest abelian extension of $\mathbb{Q}$ in $K_r$.
\end{proof}

\begin{lem}\label{k2disjoint}
$K_l \cap K_2=K$ for all $l \neq 2$.
\end{lem}

\begin{proof}
Suppose to the contrary that $l$ is a prime such that $K_2 \cap K_l = K\left(\sqrt{-1}\right)$.  Since $K\left(\sqrt{-1}\right)$ is abelian over $\mathbb{Q}$ and contained in $K_l$, by Lemma \ref{Krabelian}, $K\left(\sqrt{-1}\right) \subseteq K\left(\zeta_l\right)$.  Since $K\left(\sqrt{-1}\right)$ is the unique degree two extension of $K$ in $K\left(\zeta_l\right)$, we deduce $\displaystyle K\left(\sqrt{(-1)^{\frac{l-1}{2}} l}\right) \subseteq K\left(\sqrt{-1}\right)$.  Two cases arise.\\

\textbf{Case 1}: If $l \equiv 3 \mod 4$, then $K\left(\sqrt{-l}\right)\subseteq K\left(\sqrt{-1}\right)$.  Notice $K\left(\sqrt{-l}\right)$ is ramified over $\mathbb{Q}$ at the prime $l$.  Since $K$ contains a unit of norm $-1$, no prime $l\equiv 3 \mod 4$ can ramify in $K/\mathbb{Q}$, and so $l$ is unramified in $K\left(\sqrt{-1}\right)/\mathbb{Q}$, a contradiction.\\

\textbf{Case 2}:  If $l \equiv 1 \mod 4$, then $K\left(\sqrt{l}\right) \subset K\left(\sqrt{-1}\right)$, and the containment must be proper since $K\left(\sqrt{l}\right) \subseteq \mathbb{R}$.  This forces $\sqrt{l}\in K$, and so $2| [K\cap \mathbb{Q}\left(\zeta_l\right):\mathbb{Q}]$.  By Galois theory the isomorphism $\Gal\left(K\left(\zeta_l\right)/K\right) \cong \Gal\left(\mathbb{Q}\left(\zeta_l\right)/K\cap \mathbb{Q}\left(\zeta_l\right)\right)$ gives $\left[K\left(\sqrt{-1}\right) \cap \mathbb{Q}\left(\zeta_l\right):K \cap \mathbb{Q}\left(\zeta_l\right)\right]=2$.  This implies that $4|\left[K\left(\sqrt{-1}\right) \cap \mathbb{Q}\left(\zeta_l\right):\mathbb{Q}\right]$, a contradiction since $\Gal\left(K\left(\sqrt{-1}\right) \cap \mathbb{Q}\left(\zeta_l\right)/\mathbb{Q}\right)$ is a quotient of both the cyclic group $\Gal\left(\mathbb{Q}\left(\zeta_l\right)/\mathbb{Q}\right)$ and the elementary abelian group $\Gal\left(K\left(\sqrt{-1}\right)/\mathbb{Q}\right)$.
\end{proof}

 Let $E$ denote the genus field of $K$, the largest Abelian extension of $\mathbb{Q}$ contained in the Hilbert class field of $K$.  Define $E_k$ as the compositum $E_k=E\cdot K_k$ for $k$ odd and squarefree.  Let $\Delta$ denote the discriminant of $K/\mathbb{Q}$.

\begin{lem}\label{disjoint}
Let $r$ and $s$ denote squarefree odd integers with $\gcd(r,s)=1$.  
\begin{enumerate}\renewcommand{\labelenumi}{(\alph{enumi})}
\item $E_r \cap E_s = E$.
\item If $\gcd(r,\Delta)=1$, then $K_r \cap K_s=K$.
\end{enumerate}
\end{lem}
\begin{proof}
$(a)$: The field extensions $K_k/K$ and $E_k/E$ are ramified only at primes lying above primes $l|k$.  Therefore since $\gcd(r,s)=1$, the extension $E_r \cap E_s/E$ is unramified.  Since the extension $E/K$ is unramified, also $E_r \cap E_s/K$ is unramified.  We show $E_r \cap E_s/\mathbb{Q}$ is abelian; this implies $E_r \cap E_s$ is contained in the genus field of $K$, which is $E$, and so $E_r \cap E_s = E$.  Notice $E_r \cap E_s $ is a subfield of both $E_r(\zeta_{rs})$ and $E_s(\zeta_{rs})$, and so $E_r \cap E_s \subseteq E_r(\zeta_{rs}) \cap E_s(\zeta_{rs})$.  The extension $E_r(\zeta_{rs})$ of $E(\zeta_{rs})$ is obtained by adjoining $r^{th}$ roots, and so its degree is power of $r$; likewise $[E_s(\zeta_{rs}):E(\zeta_{rs})]$ is a power of $s$.  Thus, since $\gcd(r,s)=1$, $E_r(\zeta_{rs}) \cap E_s(\zeta_{rs}) = E(\zeta_{rs})$.  Thus $E_r \cap E_s \subseteq E(\zeta_{rs})$, and so $E_r \cap E_s$ is abelian over $\mathbb{Q}$.\\

$(b)$: Now suppose $\gcd(r,\Delta)=1$.  Then for a prime $l|r$, $K$ is unramified at $l$ and $\mathbb{Q}(\zeta_r)$ is totally ramified at $l$, and so $K(\zeta_r)=K \cdot \mathbb{Q}(\zeta_r)$ is totally ramified over $K$.  Thus the extension $K(\zeta_r) \cap E$ is both totally ramified and unramified over $K$, and so $K(\zeta_r) \cap E = K$.  Now, $K_r \cap K_s \subseteq E_r \cap E_s$, so $K_r \cap K_s$ is both abelian over $\mathbb{Q}$ and contained in $E$ by $(a)$.  By Lemma \ref{Krabelian}, this implies that $K_r \cap K_s \subseteq K(\zeta_r)$.  Hence $K_r \cap K_s \subseteq E \cap K(\zeta_r)=K$, so $K_r \cap K_s = K$.  
\end{proof}

\begin{lem}\label{ccmult}
$|C_{rs}|=|C_{r}||C_s|$ if $\gcd(r,s)=1$.
\end{lem}
\begin{proof}
It is enough to prove this for $r$, $s$ prime numbers not equal to $2$.  We show the map $C_{rs} \longrightarrow C_r \times C_s$, $\sigma \mapsto (\sigma|_{K_r}, \sigma|_{K_s})$ is a bijection of sets.  It is clearly an injection since $K_{rs}=K_r K_s$.  To see the map is a surjection, suppose $(\rho, \sigma) \in C_r \times C_s$.  $K_r \cap K_s \subseteq E_r\cap E_s$, so by Lemmas \ref{Krabelian} and \ref{disjoint}, $K_r \cap K_s$ is abelian over $\mathbb{Q}$ and contained in $K(\zeta_r) \cap K(\zeta_s)$.  By assumption $\rho=(id)$ on $K(\zeta_r)$ and $\sigma = (id)$ on $K(\zeta_s)$, and so $\rho$ and $\sigma$ agree and are the identity on $K_r \cap K_s$.  Hence there is a map $\tau:K_k \rightarrow K_k$ in $C_{rs}$ mapping to $(\rho, \sigma)$. 
\end{proof}

\begin{lem}\label{clcount}
$|C_l| \leq 2^{(n-1)} l^{n-2}$ for every $l\neq 2$.
\end{lem}

\begin{proof}
If $l \neq 2$, $l \nmid |G|$, and so $G_l \cong \oplus_{i=1}^{n-1} V_i$, where $V_i$ are distinct irreducible one-dimensional $G-$modules.  Thus the number of elements that generate $G_l$ as a $G-$module is $\displaystyle (l-1)^{n-1}$, and so $\displaystyle |C_l|=l^{n-1} - (l-1)^{n-1}$.  For any $l$, $\displaystyle (l-1)^{n-1} \geq l^{n-1} - \sum_{i=0}^{n-2} \binom{n-1}{i}  l^i$, and so 
$$\displaystyle |C_l| \leq l^{n-1} - l^{n-1} + \sum_{i=0}^{n-2}  \binom{n-1}{i}  l^i \leq  \sum_{i=0}^{n-2} \binom{n-1}{i}l^{n-2} \leq 2^{n-1}l^{n-2}.$$
\end{proof}

\begin{lem}\label{sumisfinite}
$\displaystyle \lim_{x \rightarrow \infty}  \sum_{k \in U(x)} \frac{\mu(k)|C_k|}{[K_k :
\mathbb{Q}]}  = \sum_{k=1}^{\infty} \frac{\mu (k)|C_k|}{[K_k :
\mathbb{Q}]}  < \infty$.\\
\end{lem}

\begin{proof} 

If $k \notin U(x)$, then  $k > \xi_1(x)$.  Thus 
$$ \left|\sum_{k=1}^{\infty} \frac{\mu (k)|C_k|}{[K_k :
\mathbb{Q}]} -  \sum_{k \in U(x)} \frac{\mu(k)|C_k|}{[K_k :
\mathbb{Q}]}\right| \leq \left|\sum_{k > \xi_1(x)} \frac{\mu(k)|C_k|}{[K_k :
\mathbb{Q}]}\right|= \sum_{k > \xi_1(x)} \frac{|\mu(k)||C_k|}{[K_k :
\mathbb{Q}]},$$
 and so we can prove both that the sums are equal and finite in the same step.  Let $1> \delta > 0$.  For all $k$ squarefree, 
$$ \frac{1}{[K_k :\mathbb{Q}]} \leq \frac{2^{n-1}}{ \phi(k)k^{n-1}}\text{, and so } \sum_{k > \xi_1(x)} \frac{|\mu(k)||C_k|}{[K_k :\mathbb{Q}]}\leq \sum_{k > \xi_1(x)} \frac{|\mu(k)||C_k|2^{n-1}}{ \phi(k)k^{n-1}}.$$
  By Lemmas \ref{ccmult} and \ref{clcount} $\displaystyle |C_k| = \prod_{l|k}|C_l| \leq \prod_{l|k} 2^{(n-1)} l^{n-2}$.  Thus 

$$\sum_{k > \xi_1(x)} \frac{|\mu(k)||C_k|2^{n-1}}{ \phi(k)k^{n-1}} \leq \sum_{k > \xi_1(x)} \frac{|\mu(k)| \prod_{l|k}2^{n-1}l^{n-2}2^{n-1}}{ \phi(k)k^{n-1}}= \sum_{k > \xi_1(x)}\frac{|\mu(k)| \prod_{l|k}2^{2n-2}}{\phi(k)k}.$$

We show that there is a $\kappa(\delta)$ such that  for all $k$
squarefree, 
$$ \frac{\prod_{l|k}2^{2n-2}}{\phi(k)k} \leq
\frac{\kappa(\delta)}{k^{2-\delta}}\text{, or equivalently } \frac{\prod_{l|k}2^{2n-2}}{\phi(k)} \leq
\frac{\kappa(\delta)}{k^{1-\delta}}.$$

First note that for all $l > (2^{2n-1})^{1/\delta}=:r$, $l^{\delta} > 2^n$,
and so $l-1 \geq \frac{1}{2}l>2^{2n-2}l^{1-\delta}$.  Hence $\displaystyle 1//l^{1-\delta} > 2^{2n-2}/(l-1)$
for $l >r$.  Define $\displaystyle \kappa(\delta)=2^{(2n-2)r}\prod_{l\leq
r}l^{1-\delta}$; then for all $k$ squarefree,
\begin{eqnarray*}
\frac{\prod_{l|k}2^{2n-2}}{\phi(k)}&=&\prod_{l|k}{\frac{2^{2n-2}}{l-1}} \leq \prod_{l|k,l\leq r}{\frac{2^{2n-2}}{l-1}} \cdot
\prod_{l|k,l>r}{\frac{1}{l^{1-\delta}}} \\
&\leq& \prod_{l|k}{\frac{1}{l^{1-\delta}}} \cdot(2^{(2n-2)})^r \prod_{l
\leq r}l^{1-\delta}= \prod_{l|k}\frac{1}{l^{1-\delta}} \cdot \kappa(\delta)= \frac{\kappa(\delta)}{k^{1-\delta}}.\\
\end{eqnarray*}
Since $\mu(k) =0$ for all $k$ which are not squarefree, 
$$ \sum_{k > \xi_1(x)} \frac{|\mu(k)||C_k|2^{n-1}}{ \phi(k)k^{n-1}} \leq \kappa(\delta) \sum_{k>\xi_1(x)}\frac{1}{k^{2-\delta}},$$
 which converges since $1>\delta>0$.
\end{proof}

Taking a limit and applying Lemma \ref{sumisfinite} to equation $ \mathbf{(3)}$, we conclude
$$\displaystyle{\delta(S) = \lim_{x \rightarrow \infty}
\frac{N(x)}{x/\log x}= \sum_{k=1}^{\infty} \frac{\mu(k)|C_k|}{[K_k :
\mathbb{Q}]} }.$$  

\subsection{Nonvanishing of density}

Let $$\displaystyle c= \sum_{k=1}^{\infty} \frac{\mu(k)|C_k|}{[K_k:\mathbb{Q}]}.$$ 
It remains to show that $c \neq 0$ to conclude that the set $S$ has positive density. By Lemma \ref{k2disjoint},  $K_2\cap K_l=K$ for all $l \neq 2$, and so \\
 $$ c=  \left( 1-\frac{|C_2|}{[K_2:K]} \right)\sum_{k\text{ odd}} \frac{\mu(k)|C_k|}{[K_k:\mathbb{Q}]}=\frac{1}{2}\sum_{k\text{ odd}} \frac{\mu(k)|C_k|}{[K_k:\mathbb{Q}]}.$$
Using Theorem \ref{reformulationpowerof2} we conclude $\displaystyle 2c=\sum_{k\text{ odd}} \frac{\mu(k)|C_k|}{[K_k:\mathbb{Q}]}$ is the density of primes $p \in \mathbb{Z}$ completely split in $K$ for which the index of  $\psi_p(\mO_K^{\times})$ in $\left( \mO_K/p\mO_K\right)^{\times}$ is a power of 2 times $ (p-1)/2$.  Define $d$ to be the density of primes $p \in \mathbb{Z}$ completely split in $E$ for which the index of $\psi_p(\mO_K^{\times})$ in $\left( \mO_K/p\mO_K\right)^{\times}$ is a power of 2 times $ (p-1)/2$.  Since every prime that splits in $E$ is already split in $K$, $d \leq 2c$.  We show $d >0$ to conclude that $c >0$.\\

\begin{lem}\label{ramnot2}
For any $p \neq 2$, the ramification index  $e_p(K/\mathbb{Q})$ is at most $2$.
\end{lem}

\begin{proof}
Let $p \neq 2$, and complete the extension $K/\mathbb{Q}$ at a prime above $p$.  Let $\hat{G}$  denote the Galois group of the completed extension, $\hat{G}_{0}$ denote the inertia group, and $\hat{G}_i$ denote the higher ramification groups for $ i \geq 1$.  Then $\hat{G}$  is multiquadratic, and $\hat{G}_0 /\hat{G}_1$ is cyclic (being isomorphic to a multiplicative subgroup of a finite field) and multiquadratic, hence of order 1 or 2.   Since the higher ramification groups satisfy $\hat{G}_i/\hat{G}_{i+1} \hookrightarrow k$ for the residue field $k$ of order a power of $p \neq 2$, we deduce $|\hat{G}_i /\hat{G}_{i+1}|=1$ for all $i \geq 1$.   Since the higher ramification groups form an exhaustive filtration of $\hat{G}$, we further deduce that $|\hat{G}_1| =1$.  Because $|\hat{G}_0 /\hat{G}_1| =1$ or 2, we conclude that $|\hat{G}_0|=1$ or 2, and thus the ramification index $e_p(K/\mathbb{Q})= |\hat{G}_0|=1$ or 2. 
\end{proof}

\begin{lem}\label{genusismq}
The genus field $E$ of $K$ is a multiquadratic field.
\end{lem}

\begin{proof}
Since the genus field $E$ is abelian over $\mathbb{Q}$, we prove $E$ is multiquadratic field by showing that every cyclic quotient of $\Gal(E/\mathbb{Q})$ of prime power order is of order $2$.  Let  $\Gal(E/\mathbb{Q}) \twoheadrightarrow \mathbb{Z}/l^n \mathbb{Z}$ for some prime $l$ and some $n \geq 1$.  By Galois theory, there is a subfield $R$ of $E$ with $\Gal(R/\mathbb{Q}) \cong \mathbb{Z}/l^n \mathbb{Z}$.  At least one prime must be totally ramified in $R/\mathbb{Q}$: to see this, suppose not.  Then the inertia group of every prime is proper in $\Gal(R/\mathbb{Q})$ and so is a subgroup of $H$, the unique subgroup of $\Gal(R/\mathbb{Q})$ of order $l^{n-1}$.  The fixed field $R^H$ of $H$ is then a field extension of degree $l$ over $\mathbb{Q}$ that is unramified at every prime, a contradiction.\\

Let $p$ be a prime that is totally ramified in $R/\mathbb{Q}$.  Then the ramification index $e_p(R/\mathbb{Q})$ equals $[R:\mathbb{Q}]$. By Lemma \ref{ramnot2} and the assumption that $e_2(K/\mathbb{Q}) \leq 2$, we deduce that $e_p(K/\mathbb{Q}) \leq 2$. Since ramification indices are multiplicative, $e_p(K/\mathbb{Q}) \leq 2$, and the extension $E/K$ is unramified at every prime, we conclude that $e_p(R/\mathbb{Q}) \leq e_p(E/\mathbb{Q}) \leq 2$.  Hence it must be that $e_p(R/\mathbb{Q}) = 2 = [R:\mathbb{Q}]$.
\end{proof}

For any odd integer $k \geq 1$, let $C'_k$ denote the union of conjugacy classes in $\Gal(E_k
/\mathbb{Q})$ given by
$$C'_k= \bigcup \left\{\sigma \in \Gal(E_k/\mathbb{Q}): \sigma|_{E}=(id), \sigma|_{\mathbb{Q}(\zeta_l)}=(id) \text{ and }\right.$$
$$ \left. \sigma|_{K_l} \text{ does not generate } G_l \text{ as a }
G-\text{module for all }  l|k \right\}.$$
Using Lemma \ref{genusismq}, we deduce from a symmetric argument to the one provided that $\displaystyle d=\sum_{k\text{ odd}} \frac{\mu(k)|C'_k|}{[E_k:\mathbb{Q}]}$.  

\begin{lem}\label{ciscprime}
If $k \geq 1$ is odd, $|C_k|=|C'_k|$.  
\end{lem}

\begin{proof}
We have the following diagram.\\
\begin{center}
$ \xymatrix{    1 \ar[d] \ar[r] &  \Gal(K_k/K(\zeta_k))  \ar[d]  \ar[r] & \Gal(K_k/K) \ar[d] \ar[r] & \Gal(K(\zeta_k)/K) \ar[d] \ar[r] & 1 \ar[d] \\ 
  1 \ar[r]  & \Gal(E_k/E(\zeta_k)) \ar[r] & \Gal(E_k/E) \ar[r] & \Gal(E(\zeta_k)/E) \ar[r] & 1}$ \\
  \end{center}

\noindent Since $E/K$ is unramified, $\Gal(E(\zeta_k)/E) \cong \Gal(K(\zeta_k)/K)$, and for any $l|k$, an element of $\Gal(K_k/K)$ that is the identity on $E$ generates $G_l$ as a $G-$module if and only if its image in $\Gal(E_k/E)$ also generates $G_l$ as a $G-$module.  The result follows.
\end{proof}

By Lemmas \ref{disjoint} and \ref{ccmult} we deduce $d$ is completely multiplicative.  Combining this with Lemma \ref{ciscprime}, we have
\begin{eqnarray*} 
d=\sum_{k\text{ odd}} \frac{\mu(k)|C_k|}{[E_k:\mathbb{Q}]} &=& \frac{1}{[E:\mathbb{Q}]} \sum_{k\text{ odd}} \frac{\mu(k)|C_k|}{[E_k:E]}\\
 &=& \frac{1}{[E:\mathbb{Q}]} \prod_{l \neq 2\text{ prime}} \left(1+ \frac{\mu(l)|C_l|}{[E_l:E]}\right)\\
&=& \frac{1}{[E:\mathbb{Q}]} \prod_{l \neq 2\text{ prime}} \left(1 - \frac{|C_l|}{[E_l:E]}\right)\\
  &=& \frac{r_K}{[E:\mathbb{Q}]} \prod_{l\neq 2\text{ prime}} \left(1 - \frac{|C_l|}{(l-1)l^{n-1}}\right),\\
 \end{eqnarray*}
where $r_K \neq 0$ is a correction factor for those primes ramifying in $K/\mathbb{Q}$. We show this product is nonzero by proving the corresponding series 
$$ \log \left(\prod_{l \text{ prime}}  \left(1 - \frac{|C_l|}{(l-1)l^{n-1}}\right)\right)= \sum_l \log  \left(1 - \frac{|C_l|}{(l-1)l^{n-1}}\right)$$ is absolutely convergent.  By elementary comparisons we deduce $$ \sum_{l \text{ prime}} \log  \left(1 - \frac{|C_l|}{(l-1)l^{n-1}}\right) \leq \sum_{l \text{ prime}}  \frac{|C_l|}{(l-1)l^{n-1}}.$$  
Applying Lemma \ref{clcount}, 
$$\sum_{l \text{ prime}}  \frac{|C_l|}{(l-1)l^{n-1}}\leq \sum_{l \text{ prime}}  \frac{2^{n-1}l^{n-2}}{(l-1)l^{n-1}}=\sum_{l \text{ prime}}  \frac{2^{n-1}}{(l-1)l} \leq 2^{n-1} \sum_r \frac{1}{(r-1)r} <\infty. $$
 We conclude $d >0$, so $S$ has a positive density in the set of all primes.  Since also $[K(\zeta_l):K]= l-1$ for all $l \nmid \Delta$, we have shown that the density of $S$ equals
$$ c= \frac{1}{2[K:\mathbb{Q}]} \left( \sum_{k|\Delta} \frac{\mu(k)|C_k|}{[K_k:K]} \right)\prod_{l \nmid 2\Delta}  \left(1 - \frac{1}{(l-1)}\left(1-\frac{(l-1)^{n-1}}{l^{n-1}}\right)\right).$$
We have proven the following theorem.

\begin{thm}\label{maintheorem}

Let $K$ be a totally real multiquadratic field with unit group $\mO_K^{\times}$ that satisfies $e_2(K/\mathbb{Q}) \leq 2$.  The generalized Riemann hypothesis implies that there is a positive
density $c$ of primes split in $K/\mathbb{Q}$ for which the image of $\mO_K^{\times}$ in $\left( \mO_K/p\mO_K\right)^{\times}$ has index $ (p-1)/2$ if and only if $K$ contains a unit of norm $-1$.   Moreover, when $K$ contains a unit of norm $-1$, the density $c$ is given by
$$ c= \frac{1}{2[K:\mathbb{Q}]} \left( \sum_{k|\Delta} \frac{\mu(k)|C_k|}{[K_k:K]} \right)\prod_{l \nmid 2\Delta}  \left(1 - \frac{1}{(l-1)}\left(1-\frac{(l-1)^{n-1}}{l^{n-1}}\right)\right).$$
\end{thm}

\section{Units in multiquadratic fields}\label{examples}
Recall that $K$ is a totally real multiquadratic field containing a unit of norm $-1$ of degree $n=2^m$ over $\mathbb{Q}$. We call a set of $n-1$ generators for the non-torsion summand of the unit group a \textit{system of fundamental units} of $K$.  The question of which multiquadratic fields contain a unit of norm $-1$ has been extensively studied in the literature.  Kuroda's class number formula for multiquadratic fields, see \cite{Kr} or \cite{Wa}, relates units of $K$ with units  and class numbers of subfields.  \\

\begin{thm}\label{Kuroda} \textbf{(Kuroda's class number formula)}
Let $K$ denote a totally real multiquadratic field with $|\Gal(K/\mathbb{Q})|=n=2^m$.  Let $Q_i$, $i=1, \ldots ,n- 1$ denote the $n-1$ distinct quadratic subfields of $K$.  Let $h_i$ denote the class number of $Q_i$, and let $\mO_{Q_i}^{\times}$ be the unit group of $Q_i$.  Let $v=m(2^{m-1} - 1)$, and let $h$ denote the class number of $K$.  Then 
$$ h = \frac{1}{2^v} \left[ \mO_K^{\times} : \prod^{n-1}_{i=1} \mO_{Q_i}^{\times} \right] \cdot \prod^{n-1}_{i=1} h_i.$$
\end{thm}

\begin{pro}\label{proposition:subfield}
If $K$ is a number field containing a unit of norm $-1$, then every subfield $F$ of $K$ also contains a unit of norm $-1$.
\end{pro}
\begin{proof}
Let $ u \in \mO_K^{\times}$ be a unit such that $N_{\mathbb{Q}}^K u = -1$, and let $F$ be a subfield of $K$.  Then $x=N_F^K u$ is an element of $F$, and $N_{\mathbb{Q}}^{F} x= N_{\mathbb{Q}}^{F} N_{F}^K u =  N_{\mathbb{Q}}^K u=-1$.  Hence $F$ contains a unit $x$ of norm $-1$.
\end{proof}

 Thus we consider units in (multiquadratic) subfields.  If a quadratic field $K=\mathbb{Q}(\sqrt{d})$ contains a unit of norm $-1$, then all prime divisors of $d$ are either $2$ or equivalent to $1 \mod 4$.  If $K$ is a totally real biquadratic field, choose positive integers $m$ and $n$ such that $K = \mathbb{Q}( \sqrt{m}, \sqrt{n})$.  Call the fundamental units in the three quadratic subfields $Q_1 = \mathbb{Q}(\sqrt{m})$, $Q_2 = \mathbb{Q}(\sqrt{n})$, and $Q_3 = \mathbb{Q}(\sqrt{mn})$ $\epsilon_1$, $\epsilon_2 $ and $\epsilon_3 $, respectively.  The structure of the unit group of a totally real biquadratic field has been extensively studied, and in \cite{Kb} (see also \cite{Mouhib}, \cite{Wu}), Kubota completely classified the possible structures into one of seven types.  \\

Kubota's work proves that if each quadratic subfield of $K$ has a unit $x_{i} \in Q_i$ of norm $N_{\mathbb{Q}}^{Q_i} x_i = -1$, then a system of fundamental units of $K$ must be of one of the two forms $\left\{ \epsilon_1 , \epsilon_2 , \epsilon_{3} \right\}$ or $ \left\{ \epsilon_1 , \epsilon_2 , \sqrt{\epsilon_1 \epsilon_2 \epsilon_{3}} \right\}$.  The latter occurs if and only if the element $ z = \sqrt{\epsilon_1 \epsilon_2 \epsilon_{3}} $ lies in $K$, and $z$ will be a unit of $K$ of norm $-1$, as desired. 

\begin{pro}
Suppose $S=\left\{ \epsilon_1 , \epsilon_2 , \epsilon_{3} \right\}$ is a system of fundamental units for $K$, where each $\epsilon_i$ is the fundamental unit of the quadratic subfield $Q_i$.  Then every unit of $K$ has norm +1. 
\end{pro}
\begin{proof}
Let $u \in \mO_{K}^{\times}$.  Then since $S$ is a system of fundamental units for $K$, we know $u= \pm \epsilon_{1}^{a_{1}} \epsilon_{2}^{a_{2}} \epsilon_{3}^{a_{3}}$ for some integers $a_{i}$.  Hence
\begin{eqnarray*} 
 N^{K}_{\mathbb{Q}}(u) &=& N^{K}_{\mathbb{Q}}(\pm \epsilon_{1}^{a_{1}} \epsilon_{2}^{a_{2}} \epsilon_{3}^{a_{3}})  = \prod_{i=1}^{3} (N^{K}_{\mathbb{Q}}\epsilon_{i})^{a_{i}} \\
 &=& \prod_{i=1}^{3} (N^{Q_{i}}_{\mathbb{Q}}(N^{K}_{Q_{i}}\epsilon_{i}))^{a_{i}}  = \prod_{i=1}^{3} (N^{Q_{i}}_{\mathbb{Q}}(\epsilon_{i})^{2})^{a_{i}} \\
 &=& \prod_{i=1}^{3} ((N^{Q_{i}}_{\mathbb{Q}}(\epsilon_{i}))^{2})^{a_{i}}  = \prod_{i=1}^{3} ((\pm 1)^{2})^{a_{i}} = 1
\end{eqnarray*} 
 because $\epsilon_{i} \in Q_{i}$.  Thus every element in $K$ has norm 1. 
\end{proof}

\begin{cor}\label{corollary:biquad}
If $K$ is a totally real biquadratic field containing a unit of norm $-1$, then a system of fundamental units of $K$ is $ \left\{ \epsilon_1 , \epsilon_2 , \sqrt{\epsilon_1 \epsilon_2 \epsilon_{3}} \right\}$, where each $\epsilon_i$ is the fundamental unit of the quadratic subfield $Q_i$.
\end{cor}

It is known (see \cite{IR} Ch. 17, \cite{Kb}) that for primes $p,q \equiv 1 \mod 4$ with Legendre symbol $\left( \frac{p}{q} \right) = -1$, the fields $\mathbb{Q}(\sqrt{p})$, and $\mathbb{Q}(\sqrt{pq})$ contain a unit of norm $-1$. 

\begin{pro}
If $p,q \equiv 1 \mod 4$ are distinct primes with $\left( \frac{q}{p} \right) = -1$, then $K= \mathbb{Q}(\sqrt{p}, \sqrt{q})$ contains a unit of norm $-1$.
\end{pro}
\begin{proof}
Let $K$ be as given and let $Q_1 = \mathbb{Q}(\sqrt{p})$,  $Q_2=\mathbb{Q}(\sqrt{q})$, and $Q_3=\mathbb{Q}(\sqrt{pq})$ be the three quadratic subfields of $K$.  Since each $Q_i$ has fundamental unit $\epsilon_i$ of norm $-1$, from Kubota's work, we know the unit group of $K$ has one of the two systems of fundamental units $\left\{ \epsilon_1, \epsilon_2, \epsilon_3 \right\}$ or $\left\{ \epsilon_1, \epsilon_2, \sqrt{\epsilon_1 \epsilon_2 \epsilon_3} \right\}$.  $K$ has the former structure if and only if $\left[ \mO_K^{\times} : \prod_{i=1}^3 \mO_{Q_i}^{\times} \right]=1$ and $K$ has no unit of norm $-1$; $K$ has the latter structure if and only if $\left[ \mO_K^{\times} : \prod_{i=1}^3 \mO_{Q_i}^{\times} \right] = 2$ and $z=\sqrt{\epsilon_1 \epsilon_2 \epsilon_3}$ is a unit in $K$ of norm $-1$.  We will show that a system of fundamental units for $K$ is the latter structure, so $K$ contains a unit $z$ of norm $-1$.\\

Let $h_i$ denote the class number of $Q_i$.  Since $p,q \equiv 1 \mod 4$, from genus theory we know that $h_1$ and $h_2$ are both odd, and that $K=E=E^{(+)}$ is the genus field of $Q_3$ so $ 2 |h_3$.  Kuroda's class number formula for biquadratic fields yields $|C_K| =1/4 \cdot \left[ \mO_K^{\times} : \prod_{i=1}^3 \mO_{Q_i}^{\times} \right] h_1 h_2 h_3 \in \mathbb{N}$, so since $h_1$ and $h_3$ are odd, it must be that either $4 | h_3$ or $2 |\left[ \mO_K^{\times} : \prod_{i=1}^3 \mO_{Q_i}^{\times} \right] $.  We show $4 \nmid h_3$, so it must be that $2 |\left[ \mO_K^{\times} : \prod_{i=1}^3 \mO_{Q_i}^{\times} \right]$.  Then the unit group of $K$ must have system of fundamental units given by $\left\{ \epsilon_1, \epsilon_2, \sqrt{\epsilon_1 \epsilon_2 \epsilon_3} \right\}$, so $K$ contains a unit of norm $-1$.  \\

Let $H$ denote the Hilbert class field of $Q_3$ and $C_3= \Gal(H/Q_3)$ denote the class group of $Q_3$.  To see that $4 \nmid h_3$, suppose to the contrary that $4|h_3= |\Gal(H/Q_3)|$.  Since $H/Q_3$ is abelian, $Q_3$ must have an unramified field extension of degree $4$, say $\tilde{H}$.  Since the genus field of $Q_3$ is $K=\mathbb{Q}(\sqrt{p}, \sqrt{q})$, $\Gal(E/Q_3)\cong \mathbb{Z}/2\mathbb{Z} \cong C_3 / C_3^2$, and so there is a unique summand of $C_3$ of 2-power order, so that the 2-part of $C_3$ is cyclic.  Thus $Q_3 \subseteq K \subseteq \tilde{H}$ and $\Gal(\tilde{H}/Q_3) \cong \mathbb{Z} / 4 \mathbb{Z}$. Also, we see that $\tilde{H}/\mathbb{Q}$ is Galois, since $H/\mathbb{Q}$ is Galois and $\tilde{H}$ is the unique subfield of $H$ of degree $8$ over $\mathbb{Q}$.  $\tilde{H}/\mathbb{Q}$ will not be abelian, since $E$ is the largest extension of $K$ that is Galois over $\mathbb{Q}$, so $\Gal(\tilde{H}/\mathbb{Q})$ is a nonabelian group of order $8$.  We next prove it is isomorphic to $D_8$. \\

\begin{lem}
$G :=\Gal(\tilde{H}/\mathbb{Q}) \cong D_8$, the dihedral group of order $8$.
\end{lem}
\begin{proof}
$G$ fits into the exact sequence 
$$ 1 \rightarrow \Gal(\tilde{H}/Q_3) \rightarrow G \rightarrow \Gal(Q_3/\mathbb{Q}) \rightarrow 1. $$
Let $I_p$ be the inertia group at the prime $p$.  $I_p$ is a subgroup of $G$  of order 2, since $p$ ramifies in $Q_3 /\mathbb{Q}$ and is unramified in the rest of the extension $\tilde{H}/Q_3$.  Thus $I_p \cong \Gal(Q_3/\mathbb{Q})$, which gives a splitting of the above sequence.  Thus
 $G$ is isomorphic to the semidirect product of $\mathbb{Z}/4\mathbb{Z}$ and $\mathbb{Z}/2\mathbb{Z}$.  Let $\sigma \in G$ be an element which maps to $-1$ in $\Gal(Q_3/\mathbb{Q})$, so that $\sigma$ does not fix $Q_3$.  Then we can use class field theory to compute the action of $\sigma$  on the class group of $Q_3$, which determines the action of $\sigma$ on $G$.  For any ideal $I$ of $Q_3$, $I\sigma(I)= (N(I))$, which is a principal ideal and hence trivial in the class group.  Thus in the class group, $[ I]= [\sigma(I)]^{-1}$, so $\sigma$ acts by $-1$.  Hence $G$ is dihedral.
\end{proof}

Now we know that $G=\Gal(\tilde{H}/\mathbb{Q})$ has the presentation $\langle  a,b |a^4 =1, b^2=1, bab=a^{-1}\rangle$  since $G$ is dihedral.  $\langle a \rangle $ is the unique cyclic subgroup of $D_8$ of order 4, so $\langle a \rangle = \Gal(\tilde{H}/Q_3) \leq \Gal(\tilde{H}/\mathbb{Q})$.  Let $\p = (p, \sqrt{pq})$ be a prime ideal of $Q_3$ lying above $p$, and let $\Q$ be a prime in $K$ lying above $\p$, and $\B$ a prime of $\tilde{H}$ lying above $\Q$.  Let $\sigma = (\p, \tilde{H}/Q_3)=(\B, \tilde{H}/Q_3)=(\p, H/Q_3)|_{\tilde{H}} \in \langle a \rangle = \Gal(\tilde{H}/Q_3)$ be the Frobenius element for the prime $\B$, or equivalently for $\p$ since the extension is abelian.  $\sigma$ has order 2 in the class group $C_3=\Gal(H/Q_3)$ since $\sigma ^2 =(p, \sqrt{pq})^2=(p,pq)=(p)$ is principal, so $\sigma$ has order 1 or 2 in $\Gal(\tilde{H}/Q_3)$, meaning $\sigma=a^2$ or $1$.  Now, $\Gal(\tilde{H}/Q_3)\cong D_8$ has three subgroups of order 4, namely $\langle a \rangle =\Gal(\tilde{H}/Q_3)$, $\Gal(\tilde{H}/\mathbb{Q}(\sqrt{p}))$, and $\Gal(\tilde{H}/\mathbb{Q}(\sqrt{q}))$, and each of them contain the elements $1$ and $a^2$.  Thus $\sigma \in \Gal(\tilde{H}/Q_3) \cap \Gal(\tilde{H}/\mathbb{Q}(\sqrt{p})) =\Gal(\tilde{H}/Q_3 \cdot \mathbb{Q}(\sqrt{p})) =\Gal(\tilde{H}/K)$ by Galois theory.  Hence $\sigma$ fixes the field $K$.  Now, $\sigma=(\B , \tilde{H}/Q_3)$ is the Frobenius element attached to a prime $\B$ in $\tilde{H}$ lying above $p$.  This implies that the Frobenius element $(\Q, K/\mathbb{Q}(\sqrt{p}))= (id)$ as well, since $\Q | p$.  But $K=(\mathbb{Q}(\sqrt{p}))(\sqrt{q})$, and so we know that the Frobenius symbol in this field extension is given by the Legendre symbol.  That is, for $\Q|p$, $(\Q, K/\mathbb{Q}(\sqrt{p}))=1$ if and only if $\left(\frac{q}{p} \right) = +1$.  This contradicts our initial assumption that $\left(\frac{q}{p} \right) = -1$.  Hence it is not the case that $4|h_3$, and so $K$ contains a unit of norm $-1$.
\end{proof}

\begin{exa}\label{5-13norm-1}
 Consider the field $K=\mathbb{Q}(\sqrt{5}, \sqrt{13})$, which contains a unit of norm $-1$ and has genus field $E=K$.  Only the primes 5 and 13 ramify in $K/\mathbb{Q}$, and for these primes $[K(\zeta_l):K]=(l-1)/2$.  Using Theorem \ref{maintheorem}, we compute the density $D$ of primes for which $K$ has minimal index mod p to be

$$ D=\frac{1}{4}\cdot\frac{1}{2} \prod_{l \neq 2} \left(1- \frac{1}{[K(\zeta_l):K]}\left(1-\frac{(l-1)^3}{l^3}\right)\right)= \frac{1}{4}\cdot \frac{1}{2}\cdot \frac{35}{54}\cdot \frac{189}{250} \cdot \frac{1931}{2058} \cdots = 0.05142184\ldots.$$

The accuracy of this numerical approximation can be verified following the method given in \cite{Moree}, which gives error bounds on the log value of
$$\prod_{k=2}^{\infty} \zeta(k)^{-e_k}$$   and proves that this product is an equivalent form of the product given above.  Using PARI, we computed that the density among the first two hundred thousand primes for which $K$ has minimal index mod $p$ is

$$ 0.05176. $$
\end{exa}

\begin{lem}\label{mqprime}
If $K=\mathbb{Q}(\sqrt{d_{1}}, \ldots ,\sqrt{d_{m}})$ has odd class number, then in fact each $d_i=p_i$ is a prime number with $p_i \equiv 1 \mod 4$ for each $i$, except possibly $d_1=2$.
\end{lem}
\begin{proof}
Let $p_1, \ldots , p_s$ denote all of the prime factors of the various $d_i$, with $p_1 = 2$ or an odd prime congruent to $1 \mod 4$, and all other $p_k \equiv 1 \mod 4$.  Since by assumption $ d_i\nmid d_j$ if $i \neq j$, we know $s \geq m$, with equality holding if and only if each $d_i= p_i$ is a prime number.  Now, if  $ s > m$, then the field $F=\mathbb{Q}( \sqrt{p_1} , \ldots , \sqrt{p_s} )$ is a proper field extension of $K$ of degree $2^s$ contained in the genus field $E$ of $K$.  Hence $F \subseteq H$, the Hilbert class field of $K$, and so $2 | \left|C_K\right|$.  This is impossible by assumption, so it must be that $s=m$ and each $d_i =p_i$ is a prime number, as claimed.
\end{proof}

\begin{thm}\label{thm:triodd}
If $L$ is a totally real triquadratic field with odd class number, then every unit of $L$ has norm $+1$.
\end{thm}

\begin{proof}
 Assume to the contrary that $L$ is a totally real triquadratic field with odd class number which contains a unit $\epsilon$ of norm $N^{L}_{\mathbb{Q}}(\epsilon)= -1$. \\

\begin{lem}\label{lemma:biquadstructure} If $L$ has an element of norm $-1$, then every biquadratic subfield $B$ of $L$ has unit group structure given by  $ \{ \epsilon_{1}, \epsilon_{2}, \sqrt{\epsilon_{1} \epsilon_{2} \epsilon_{3}} \} $, where $\epsilon_{i}, 1 \leq i \leq 3$ are the fundamental units in each of the three quadratic subfields of $B$.
\end{lem}
\begin{proof}
Since $L$ has an element of norm $-1$, by Proposition ~\ref{proposition:subfield} every biquadratic subfield $B$ of $L$ also contains a unit $u$ of norm $N_{\mathbb{Q}}^B u =-1$.  Hence by Corollary ~\ref{corollary:biquad} we know $B$ must have a system of fundamental units given by $ \{ \epsilon_{1}, \epsilon_{2}, \sqrt{\epsilon_{1} \epsilon_{2} \epsilon_{3}} \} $, where $\epsilon_{1}, \epsilon_{2}$, and $\epsilon_{3}$ are the fundamental units in each of the three quadratic subfields of $B$.
\end{proof}

The structure of the unit group of $L$ must be further investigated, so we look more explicitly at $L$ to determine this structure.  $L= \mathbb{Q}( \sqrt{m}, \sqrt{n}, \sqrt{d})$ contains the seven quadratic subfields $Q_{1}=\mathbb{Q}(\sqrt{m})$, $Q_{2}=\mathbb{Q}(\sqrt{n})$, $Q_{3}=\mathbb{Q}(\sqrt{d})$, $Q_{4}=\mathbb{Q}(\sqrt{mn})$, $Q_{5}=\mathbb{Q}(\sqrt{md})$, $Q_{6}=\mathbb{Q}(\sqrt{nd})$, $Q_{7}=\mathbb{Q}(\sqrt{mnd})$.  Let $\epsilon_{r}$ denote the fundamental unit of the quadratic field $\mathbb{Q}(\sqrt{r})$ for each of the quadratic subfields listed above.  In addition, we know that $L$ contains the seven biquadratic subfields $B_{1}=\mathbb{Q}(\sqrt{m}, \sqrt{n})\ni \alpha_1 =\sqrt{\epsilon_m\epsilon_n\epsilon_{mn}}$, $B_{2}=\mathbb{Q}(\sqrt{m}, \sqrt{d})\ni \alpha_2 =\sqrt{\epsilon_m\epsilon_d\epsilon_{md}}$, $B_3=\mathbb{Q}(\sqrt{n}, \sqrt{d})\ni \alpha_3 =\sqrt{\epsilon_n\epsilon_d\epsilon_{nd}}$, $B_4=\mathbb{Q}(\sqrt{m}, \sqrt{nd})\ni \alpha_4 =\sqrt{\epsilon_m\epsilon_{nd}\epsilon_{mnd}}$, $B_{5}=\mathbb{Q}(\sqrt{n}, \sqrt{md})\ni \alpha_5 =\sqrt{\epsilon_n\epsilon_{md}\epsilon_{mnd}}$, $B_{6}=\mathbb{Q}(\sqrt{d}, \sqrt{mn})\ni \alpha_6 =\sqrt{\epsilon_d\epsilon_{mn}\epsilon_{mnd}}$, and $B_7 =\mathbb{Q}(\sqrt{mn}, \sqrt{md})\ni \alpha_7 =\sqrt{\epsilon_{mn}\epsilon_{md}\epsilon_{nd}}$.  Since $L$ has a unit of norm $-1$, by Lemma \ref{lemma:biquadstructure}  we must have that for each of the biquadratic fields $B_{i}$ of $L$ listed above, the element $\alpha_{i}=\sqrt{\epsilon_{r} \epsilon_{s} \epsilon_{t}} \in \mO_{B_{i}}^{\times}\subseteq \mO_{L}^{\times}$ for the appropriate choice of numbers $r,s$, and $t$.  \\

Since $L$ is totally real of degree $2^3$, we have that $v= 3(2^2 - 1)=9$ and so Kuroda's class number formula yields 
$$|C_{L}| = \frac{1}{2^{9}} \cdot \left[ \mO_{L}^{\times} : \prod_{i=1}^{7} \mO_{Q_{i}}^{\times} \right] \cdot \prod_{i=1}^{7} h_{i}.$$ 
 By definition $\prod_{i=1}^{7} \mO_{Q_{i}}^{\times}$ is generated as a $\mathbb{Z}-$module by $\pm 1$ and the $\epsilon_{r}$, the fundamental units in quadratic subfields $Q_i$.  Note that $\alpha_i \in \mO_L^{\times}$ and $\alpha_{i}^{2} \in \prod_{i=1}^{7} \mO_{Q_{i}}^{\times}$.  We also have the inequality 
$$\left[ \mO_{L}^{\times} : \prod_{i=1}^{7} \mO_{Q_{i}}^{\times} \right] \geq \left[ \mO_{L}^{\times}/(\mO_{L}^{\times})^{2} : \prod_{i=1}^{7} \mO_{Q_{i}}^{\times}/ (\prod_{i=1}^{7} \mO_{Q_{i}}^{\times})^{2} \right],$$ 
where $\mO_{L}^{\times}/(\mO_{L}^{\times})^{2}$ and $\prod_{i=1}^{7} \mO_{Q_{i}}^{\times}/ (\prod_{i=1}^{7} \mO_{Q_{i}}^{\times})^{2}$ are modules over $\mathbb{Z}/2\mathbb{Z}$.  \\

\begin{pro}
The $\alpha_{i} \in \mO_{L}^{\times}/(\mO_{L}^{\times})^{2}$ generate a rank $4$ module over $\mathbb{Z}/2\mathbb{Z}$.
\end{pro}

\begin{proof}

Consider the module generated by the units $\alpha_i \in L$.  We write each $\alpha_i$ in terms of the seven elements $\sqrt{\epsilon_m}, \sqrt{\epsilon_n},\sqrt{\epsilon_d},\sqrt{\epsilon_{mn}},\sqrt{\epsilon_{md}}, \sqrt{\epsilon_{nd}}, \sqrt{\epsilon_{mnd}}$.  Switch to additive notation (which takes products to sums), and write $\beta_i$ to denote $\alpha_i$ in additive notation.  We want to determine the rank of the matrix over $\mathbb{Z}/2\mathbb{Z}$ whose entries are the values from the following table.  
\begin{center}
\begin{tabular}{|c|ccccccc|}
\hline
&$\sqrt{\epsilon_m}$ & $\sqrt{\epsilon_n}$ & $\sqrt{\epsilon_d}$ & $\sqrt{\epsilon_{mn}}$ & $\sqrt{\epsilon_{md}}$ & $\sqrt{\epsilon_{nd}}$& $\sqrt{\epsilon_{mnd}}$\\ \hline
$\beta_1$ & 1&1&0&1&0&0&0\\
$\beta_2$ & 1&0&1&0&1&0&0 \\
$\beta_3$ & 0&1&1&0&0&1&0 \\
$\beta_4$ & 0&0&1&1&0&0&1\\
$\beta_5$ & 0&1&0&0&1&0&1\\
$\beta_6$ & 1&0&0&0&0&1&1\\
$\beta_7$ & 0&0&0&1&1&1&0\\
\hline
\end{tabular}
\end{center}
 
By adding rows and reducing mod 2, we see that $\beta_1+\beta_2 +\beta_3 = \beta_7$, $\beta_2+\beta_4+\beta_6= \beta_7$, and $\beta_3+\beta_4+\beta_5=\beta_7$.  The rows corresponding to  $\beta_2$, $\beta_3$, $\beta_4$, and $\beta_7$ are linearly independent rows since each pivots in a different column, and so we see that this matrix has rank 4.  Hence in multiplicative notation, we have shown that the module generated by the $\alpha_i$ is minimally generated by \\$\alpha_{2}= \sqrt{\epsilon_{m}\epsilon_{d}\epsilon_{md}}$, $\alpha_{3}= \sqrt{\epsilon_{n}\epsilon_{d}\epsilon_{nd}}$, $\alpha_{4}= \sqrt{\epsilon_{m}\epsilon_{nd}\epsilon_{mnd}}$, and $\alpha_{7}= \sqrt{\epsilon_{mn}\epsilon_{md}\epsilon_{nd}}$  as a $\mathbb{Z}/2\mathbb{Z}-$ module.\\
\end{proof} 

This proposition proves $\left[ \mO_{L}^{\times}/(\mO_{L}^{\times})^{2} : \prod_{i=1}^{7} \mO_{Q_{i}}^{\times}/ (\prod_{i=1}^{7} \mO_{Q_{i}}^{\times})^{2} \right] \geq 2^{4}$,\\  which combined with the previous inequality proves $\left[ \mO_{L}^{\times} : \prod_{i=1}^{7} \mO_{Q_{i}}^{\times} \right] \geq 2^4$.\\

We use this inequality in combination with the class number formula to obtain a contradiction to the assumption that $L$ contains a unit of norm $-1$.  $L$ contains the quadratic subfield $Q_{4}=\mathbb{Q}(\sqrt{mn})$, which has a unit of norm $-1$.  Thus $\mathbb{Q}(\sqrt{m}, \sqrt{n})$ is contained in the genus field of $\mathbb{Q}(\sqrt{mn})$, and so $\mathbb{Q}(\sqrt{m}, \sqrt{n})$ is contained in the Hilbert class field of $Q_{4}$.  Hence the class number of $Q_{4}$, $h_{4}$, is divisible by $2$.  Similarly $\mathbb{Q}(\sqrt{md})$ and $\mathbb{Q}(\sqrt{nd})$  each have class number divisible by $2$, and $Q_{7}=\mathbb{Q}(\sqrt{mnd})$ has class number divisible by $4$ since $\mathbb{Q}(\sqrt{m}, \sqrt{n}, \sqrt{d})$ is a degree 4 extension of $Q_{7}$ contained in the genus field of $Q_{7}$.  Thus $\prod_{i=1}^{7} h_{i} \geq 1 \cdot 1 \cdot 1 \cdot 2 \cdot 2 \cdot 2 \cdot 4 = 2^{5}$.  Kuroda's class number formula now reads  
$$|C_{L}| = \frac{1}{2^{9}} \cdot \left[ \mO_{L}^{\times} : \prod_{i=1}^{7} \mO_{Q_{i}}^{\times} \right] \cdot \prod_{i=1}^{7} h_{i} \geq \frac{1}{2^{9}} \cdot 2^{4} \cdot 2^{5}=1.$$  

Since the class number of $L$ is odd, $2 \nmid |C_L|$ and so $ 2 \nmid \left[ \mO_{L}^{\times} : \prod_{i=1}^{7} \mO_{Q_{i}}^{\times} \right]/2^4$.  Thus it must be that 4 is the highest power of 2 dividing $\left[ \mO_{L}^{\times} : \prod_{i=1}^{7} \mO_{Q_{i}}^{\times} \right]$.  However, for any $\beta \in \mO_{L}^{\times}$, $\beta^{4}= \prod_{i=1}^{7}N^{L}_{Q_{i}}(\beta)/(N^{L}_{\mathbb{Q}}(\beta))^{3} \in \prod_{i=1}^{7} \mO_{Q_{i}}^{\times}$, and so $\mO_{L}^{\times} / (\mO_{L}^{\times})^{4}$, a group of order $2^{15}$, surjects onto $\mO_{L}^{\times} / \prod_{i=1}^{7} \mO_{Q_{i}}^{\times}$.  Thus $\left[ \mO_{L}^{\times} : \prod_{i=1}^{7} \mO_{Q_{i}}^{\times} \right]= 2^{c}$ for some $c \leq 15$.  Then since the class number of $L$ is odd, we see that $c=4$, and so the $\alpha_{i}$ generate all of the elements of $\mO_{L}^{\times}$ not in $\prod_{i=1}^{7} \mO_{Q_{i}}^{\times}$.  Now, by Lemma ~\ref{lemma:biquadstructure}, for each $i$ we know that $\alpha_i$ lie in $ \mO_{B_i}^{\times}$ for $B_i$ the biquadratic subfields of $L$.  Hence $\mO_{L}^{\times} = \prod_{i=1}^{7} \mO_{B_{i}}^{\times}$.  We now show this last assertion implies there is no unit of $L$ of norm $-1$, contradicting our assumption.  

\begin{pro}
Suppose $\mO_{L}^{\times} = \prod_{i=1}^{7} \mO_{B_{i}}^{\times}$, where the $B_i$ are the biquadratic subfields contained in $L$.  Then every element of $L$ has norm +1.
\end{pro}
\begin{proof}
Let $u \in \mO_{L}^{\times}$.  Then $u =\pm \prod_{i=1}^{7} u_{i}$, $u_{i} \in  O_{B_{i}}^{\times}$, and so 
\begin{eqnarray*} 
N^{L}_{\mathbb{Q}}(u) &=& N^{L}_{\mathbb{Q}}(\prod_{i=1}^{7} u_{i}) = \prod_{i=1}^{7} N^{L}_{\mathbb{Q}}(u_{i}) \\
&=& \prod_{i=1}^{7} N^{B_{i}}_{\mathbb{Q}}(N^{L}_{B_{i}}u_{i}) = \prod_{i=1}^{7} N^{B_{i}}_{\mathbb{Q}}(u_{i})^{2} \\
&=& \prod_{i=1}^{7} (\pm 1)^{2} = 1, 
\end{eqnarray*} 
since $u_{i} \in B_{i}$.  Hence every unit of $L$ has norm $+1$.  
\end{proof}
 
We have now proven that a triquadratic field $L$ has no unit of norm $-1$ when the class number of $L$ is odd.  
\end{proof}

We can use this result to obtain results on the existence of units of norm $-1$ for more general totally real multiquadratic fields by induction.  Let $K$ be a totally real multiquadratic number field which is Galois over $\mathbb{Q}$ of degree $\left[ K : \mathbb{Q} \right] = 2^{m}$, so that $\Gal(K/\mathbb{Q}) \cong (\mathbb{Z}/2\mathbb{Z})^{m}$ for some $m \geq 3$.  Then there are integers $d_{i} > 0$ so that $K= \mathbb{Q}(\sqrt{d_{1}}, \ldots ,\sqrt{d_{m}})$, and for all $i \neq j$, $d_i \nmid d_j$.  $K$ has $t:= 2^{m} -1$ quadratic subfields, among them the fields $\mathbb{Q}(\sqrt{d_{i}})$ as well as various quadratic fields such as $\mathbb{Q}(\sqrt{d_i d_j})$.  Denote these quadratic subfields as $Q_{1}, Q_{2}, \ldots ,Q_{t}$. In order for $K$ to contain a unit $\epsilon$ with $N^{K}_{\mathbb{Q}}(\epsilon) = -1$, by Proposition ~\ref{proposition:subfield} every subfield of $F$ of $K$ must also contain a unit of norm $-1$.  In particular, each quadratic field $Q_{i}$ of $K$ must contain a unit $u$ of norm $N_{\mathbb{Q}}^{Q_i} u = -1$.  Thus no $d_i \equiv 3 \mod 4$.  In addition, each triquadratic subfield $T$ of $K$ must contain a unit $u$ of norm $N_{\mathbb{Q}}^T u = -1$.  However, from the previous theorem we know that no triquadratic field contains a unit of norm $-1$ when the class number of $T$ is odd.  Hence we can deduce that when the class number of any triquadratic subfield $T$ of $K$ is odd, then no unit of $K$ has norm $-1$.  Next we will show that this is always case when the class number of $K$ is also odd; that is, every multiquadratic field $K$ with odd class number contains a triquadratic subfield with odd class number.\\

\begin{thm}\label{nonorm-1}
If $K=\mathbb{Q}(\sqrt{d_{1}}, \ldots ,\sqrt{d_{m}})$ is a totally real multiquadratic field of degree $n=2^m$, $m \geq 3$ and the class number of $K$ is odd, then $K$ contains no unit of norm $-1$.
\end{thm}

\begin{proof}
Let $K$ be as given. Assume to the contrary that $K$ contained a unit of norm $-1$; then no $d_i \equiv 3 \mod 4$.  We proceed in several steps.  

\begin{lem}\label{mqprime}
If $K$ has odd class number, then in fact each $d_i=p_i$ is a prime number with $p_i \equiv 1 \mod 4$ for each $i$, except possibly $d_1=2$.
\end{lem}
\begin{proof}
Let $p_1, \ldots , p_s$ denote all of the prime factors of the various $d_i$, with $p_1 = 2$ or an odd prime congruent to $1 \mod 4$, and all other $p_k \equiv 1 \mod 4$.  Since by assumption $ d_i \nmid d_j$ if $i \neq j$, we know $s \geq m$, with equality holding if and only if each $d_i= p_i$ is a prime number.  Now, if  $ s > m$, then the field $F=\mathbb{Q}( \sqrt{p_1} , \ldots , \sqrt{p_s} )$ is a proper field extension of $K$ of degree $2^s$ contained in the genus field $E$ of $K$.  Hence $F \subseteq H$, the Hilbert class field of $K$, and so $2 | \left|C_K\right|$.  This is impossible by assumption, so it must be that $s=m$ and each $d_i =p_i$ is a prime number, as claimed.
\end{proof}

Next we prove the following lemma relating class numbers of totally ramified extensions, which we will prove in a more general setting.  It is expected that this proof can be shortened.
\begin{lem}\label{totramified}
If $F \subseteq L$ is a totally ramified extension of totally real Galois number fields, then the class number of $F$ divides the class number of $L$.
\end{lem}
\begin{proof}
Let $H$ denote the Hilbert class field of $F$ and $C_F$ denote the class group of $F$, so $[H : F]= |C_F|$.  First we determine $L \cap H$.  Let $\p_0$ denote a prime of $F$ that is totally ramified in $L$.  Then $\p_0$ is totally ramified in $L \cap H \subseteq L$.  Also, since $H$ is the Hilbert class field of $F$, $H$ is an unramified extension of $F$, and so $\p_0$ is unramified in $H$ and so unramified in $L \cap H \subseteq H$.  Then we see that $\p_0$ is both unramified and ramified in $L \cap H /F$, so it must be that $F= L \cap H$.\\

Let $H_L$ denote the Hilbert class field of $L$ and $C_L$ denote the class group of $L$.  We show $HL$ is an extension of $L$ unramified at all primes, so that $HL \subseteq H_L$.  First, $HL/L$ is unramified at the infinite primes since $F$, $H$, and $L$ are all totally real.  Next, since $H \cap L = F$, by Galois theory we have an isomorphism $\Gal(HL/L) \cong \Gal (H/ H\cap L) = \Gal (H/F)$.  The isomorphism is given by restriction of an element $\sigma: HL \rightarrow HL$ to $\sigma|_H :H \rightarrow H$. Then for any $\p$ prime in $F$ and  $\Q $ prime in $L$ with $\Q | \p$, the Frobenius element at the prime $\Q$, $(\Q, HL/L)$ maps to the restriction $(\Q, HL/L)|_H$.  By properties of Frobenius elements, $(\Q, HL/L)|_H = (\p, H/F)$.  Since $H/F$ is unramified, every $(\p, H/F)$ is well defined, and so every $(\Q, HL/L)$ is well defined.  This means every ramification number $e(\Q, HL/L)$ equals one, or equivalently that each prime $\Q$ of $L$ is unramified in $HL/L$.  Thus $HL/L$ is unramified at all finite and infinite primes, so that $HL \subseteq H_L$.  Then $|C_L| = [ H_L : L] = [H_L : HL][HL:L] = [H_L: HL][H:F] = [H_L:HL] |C_F|$, since $\Gal(HL/L ) \cong \Gal(H/F)$ implies $[HL:L]=[H:F]$.  Thus we see that the class number $|C_F|$ of $F$ divides the class number of $L$, $|C_L|$. 
\end{proof}
We proceed with the proof of Theorem \ref{nonorm-1} by induction on $m$, where $[K:\mathbb{Q}]=2^m$.  We claim that for all $m \geq 3$, if the class number of $K$, a multiquadratic field of degree $2^m$ is odd, then $K$ has no unit of norm $-1$.  The base case, $m=3$, is the triquadratic case and was proven in Theorem \ref{thm:triodd}.  Now assume the induction hypothesis, so if $L$ is any multiquadratic field of degree $2^t$ with $t < m$ and odd class number, then $L$ has no unit of norm $-1$.  Consider a multiquadratic field $K$ with odd class number.  From Lemma \ref{mqprime}, $K$ is of the form $K=\mathbb{Q}(\sqrt{p_1}, \ldots ,\sqrt{p_m})$, where each $p_i$ is a prime number.  Let $K_{m-1} = \mathbb{Q}(\sqrt{p_1}, \ldots , \sqrt{p_{m-1}}) \subseteq K$, which is a multiquadratic field of degree $2^{m-1}$.  Then $K / K_{m-1}$ is a degree $2$ extension of fields that is totally ramified at primes of $K_{m-1}$ lying above $p_m$.  From Lemma \ref{totramified}, we know the class number of $K_{m-1}$ divides the class number of $K$.  Since the class number of $K$ is odd, this implies the class number of $K_{m-1}$ must also be odd.  Hence by inductive hypothesis $K_{m-1}$ contains no unit of norm $-1$, and so $K$ contains no unit of norm $-1$ by Proposition ~\ref{proposition:subfield}. 
\end{proof}

This shows that for general multiquadratic fields $K$ of degree $2^m$, $m \geq 3$ with class number $1$ (that is, the Hilbert class field $H=K$), it is never the case that there are  a positive
density of primes split in $K/\mathbb{Q}$ for which $K$ has minimal index mod $p$. It has already been shown by A. Fr\"olich \cite{Fr} that if $ m \geq 5$ then $K$ has class number $>1$, which gives an alternate argument for why this will never be true for infinitely many $p$ split in $K$.  We have proven this for $m=3,4$ as well.  
\\

It should be noted that there are multiquadratic fields of degree at least eight that do contain units of norm $-1$; one such example is $\mathbb{Q}(\sqrt{5}, \sqrt{13}, \sqrt{37})$, which has class number $2$.

\section{Application to ray class fields}\label{rayclassfields}

For a totally real Galois extension  $K$ of $\mathbb{Q}$ of degree $n$, let $H$ denote the Hilbert class field of $K$ and let $C_K$ denote the class group of $K$, so that $C_K \cong \Gal(H/K)$.  For a prime $p \in \mathbb{Z}$ which splits completely in $K$, consider the ray class field  $L_{p}$ of $K$ of conductor $(p)=p\mO_{K}$.  The corresponding ray class group $RC(p)\cong \Gal(L_{p}/K)$ of $L_{p}/K$ fits into the exact sequence
$$ \mO_{K}^{\times} \stackrel{\psi_p}{\longrightarrow} (\mO_{K}/p\mO_{K})^{\times} \longrightarrow  RC(p) \longrightarrow C_{K} \longrightarrow 1. $$
By exactness, the size of the group $RC(p)$ is $|RC(p)| = | C_{K}|\left[(\mO_{K}/p\mO_K)^{\times} : \psi_p(\mO_{K}^{\times})\right]$. Let $\zeta_p$ denote a primitive $p^{th}$ root of unity and notice $\zeta_p + \zeta_p^{-1} \in \mathbb{R}$.  Since $H({\zeta_p + \zeta_p^{-1} })$ is an abelian extension of $K$ which is ramified only at primes dividing $p\mO_{K}$,  $H({\zeta_p + \zeta_p^{-1}}) \subseteq L_{p}$.  

\begin{pro}\label{psimaximal} 
For a prime $p$ completely split in $K$, $L_{p}= H({\zeta_p + \zeta_p^{-1}})$ if and only if $K$ has minimal index mod $p$.
\end{pro}

\begin{proof}

Since $p$ splits completely in $K/\mathbb{Q}$, $(\mO_{K}/p\mO_{K})^{\times} \cong \prod_{i=1}^{n}{\mathbb{F}^{\times}_{p}}$.  Thus $\displaystyle |RC(p)|=| C_{K}| \frac{(p-1)^n}{|\psi_p(\mO_{K}^{\times})|}$ and $|Im(\psi_p)|=2(p-1)^{n-1}$ if and only if $\displaystyle |RC(p)| = | C_{K}| \frac{(p-1)^n}{|\psi_p(\mO_{K}^{\times})|}=| C_{K}| (p-1)/2$, or equivalently $\left[L_{p}: K \right]=\left[H({\zeta_p + \zeta_p^{-1}}): K \right]$.  Since $H({\zeta_p + \zeta_p^{-1}}) \subseteq L_{p}$, this happens if and only if $H({\zeta_p + \zeta_p^{-1}})= L_{p}$.
\end{proof}

We can now use Theorem \ref{maintheorem} to deduce the following.

\begin{thm}
Let $K$ be a totally real multiquadratic field that satisfies $e_2(K/\mathbb{Q}) \leq 2$, let $H$ be
the Hilbert class field of $K$ and let $\zeta_p$ be a primitive $p^{th}$ root of unity.  Then
the generalized Riemann hypothesis implies that there is a positive
density of primes completely split in $K/\mathbb{Q}$ for which the ray class field  $L_{p}$ of $K$ of conductor $(p)=p\mO_{K}$  equals $H(\zeta_p + \zeta_p^{-1})$ if and only if $K$ contains a unit of norm $-1$.  The density $c$ is given in this case by
$$ c= \frac{1}{2[K:\mathbb{Q}]} \left( \sum_{k|\Delta} \frac{\mu(k)|C_k|}{[K_k:K]} \right)\prod_{l \nmid 2\Delta}  \left(1 - \frac{1}{(l-1)}\left(1-\frac{(l-1)^{n-1}}{l^{n-1}}\right)\right).$$
\end{thm}

\section{Acknowledgments}

This article is a summary and generalization of my doctoral research at Northwestern University.   I would like to thank Frank Calegari for his never-ending guidance, support, and patience throughout this process.  I would also like to thank Matt Emerton for helpful conversations and the referee for insightful comments and suggestions.  Lastly, I would like to thank the administration and my colleagues at Birmingham-Southern College for support while I completed this project.


\begin{thebibliography}{10}

\bibitem{Ar}
E.~Artin.
\newblock {\em The collected papers of {E}mil {A}rtin}.
\newblock Edited by Serge Lang and John T. Tate. Addison--Wesley Publishing
  Co., Inc., Reading, Mass.-London, 1965.

\bibitem{CP}
L.~Cangelmi and F.~Pappalardi.
\newblock On the {$r$}-rank {A}rtin conjecture. {II}.
\newblock {\em J. Number Theory}, 75(1):120--132, 1999.

\bibitem{CW}
G.~Cooke and P.~J. Weinberger.
\newblock On the construction of division chains in algebraic number rings,
  with applications to {${\rm SL}\sb{2}$}.
\newblock {\em Comm. Algebra}, 3:481--524, 1975.

\bibitem{H}
C.~Hooley.
\newblock On {A}rtin's conjecture.
\newblock {\em J. Reine Angew. Math.}, 225:209--220, 1967.

\bibitem{IR}
K.~Ireland and M.~Rosen.
\newblock {\em A classical introduction to modern number theory}, volume~84 of
  {\em Graduate Texts in Mathematics}.
\newblock Springer-Verlag, New York, second edition, 1990.

\bibitem{Kb}
T.~Kubota.
\newblock \"{U}ber den bizyklischen biquadratischen {Z}ahlk\"orper.
\newblock {\em Nagoya Math. J.}, 10:65--85, 1956.

\bibitem{Kr}
S.~Kuroda.
\newblock \"{U}ber die {K}lassenzahlen algebraischer {Z}ahlk\"orper.
\newblock {\em Nagoya Math. J.}, 1:1--10, 1950.

\bibitem{LO}
J.~C. Lagarias and A.~M. Odlyzko.
\newblock Effective versions of the {C}hebotarev density theorem.
\newblock In {\em Algebraic number fields: {$L$}-functions and {G}alois
  properties ({P}roc. {S}ympos., {U}niv. {D}urham, {D}urham, 1975)}, pages
  409--464. Academic Press, London, 1977.

\bibitem{Le}
H.~W. Lenstra, Jr.
\newblock On {A}rtin's conjecture and {E}uclid's algorithm in global fields.
\newblock {\em Invent. Math.}, 42:201--224, 1977.

\bibitem{KRM}
K.~R. Matthews.
\newblock A generalisation of {A}rtin's conjecture for primitive roots.
\newblock {\em Acta Arith.}, 29(2):113--146, 1976.

\bibitem{Moree}
P.~Moree.
\newblock Approximation of singular series and automata.
\newblock {\em Manuscripta Math.}, 101(3):385--399, 2000.
\newblock With an appendix by Gerhard Niklasch.

\bibitem{Mouhib}
A.~Mouhib.
\newblock On the parity of the class number of multiquadratic number fields.
\newblock {\em J. Number Theory}, 129(6):1205--1211, 2009.

\bibitem{Mu}
M.~R. Murty.
\newblock On {A}rtin's conjecture.
\newblock {\em J. Number Theory}, 16(2):147--168, 1983.

\bibitem{R1}
H.~Roskam.
\newblock A quadratic analogue of {A}rtin's conjecture on primitive roots.
\newblock {\em J. Number Theory}, 81(1):93--109, 2000.

\bibitem{R2}
H.~Roskam.
\newblock Artin's primitive root conjecture for quadratic fields.
\newblock {\em J. Th\'eor. Nombres Bordeaux}, 14(1):287--324, 2002.

\bibitem{Me}
M.~E. Stadnik.
\newblock A multiquadratic field generalization of artin's conjecture,
  unabridged.
\newblock {\em Ph.D. Thesis, Northwestern University}, pages 1--110, 2012.

\bibitem{Wa}
H.~Wada.
\newblock On the class number and the unit group of certain algebraic number
  fields.
\newblock {\em J. Fac. Sci. Univ. Tokyo Sect. I}, 13:201--209 (1966), 1966.

\bibitem{cyclo}
L.~C. Washington.
\newblock {\em Introduction to cyclotomic fields}, volume~83 of {\em Graduate
  Texts in Mathematics}.
\newblock Springer-Verlag, New York, second edition, 1997.

\bibitem{We}
P.~J. Weinberger.
\newblock On {E}uclidean rings of algebraic integers.
\newblock In {\em Analytic number theory ({P}roc. {S}ympos. {P}ure {M}ath.,
  {V}ol. {XXIV}, {S}t. {L}ouis {U}niv., {S}t. {L}ouis, {M}o., 1972)}, pages
  321--332. Amer. Math. Soc., Providence, R. I., 1973.

\bibitem{Wu}
Q.~Wu.
\newblock Computing fundamental units in bicyclic biquadratic global fields.
\newblock {\em J. Ramanujan Math. Soc.}, 23(4):357--380, 2008.

\end{thebibliography}
\end{document}